\documentclass[12pt]{amsart}
\usepackage{amssymb,times,mathabx, bbm,bm,color,amsmath}

\parskip=\medskipamount

\newtheorem*{Thm*}{Theorem}
\newtheorem{Thm}{Theorem}
\newtheorem{Cor}[Thm]{Corollary}
\newtheorem{Prop}[Thm]{Proposition}
\newtheorem{Lemma}[Thm]{Lemma}

\theoremstyle{definition}
\newtheorem{Defn}[Thm]{Definition}
\newtheorem{Notation}[Thm]{Notation}

\newtheorem{Remark}[Thm]{Remark}
\newtheorem{Example}[Thm]{Example}

\newcommand{\mf}[1]{\mathbb{#1}}
\newcommand{\mc}[1]{\mathcal{#1}}
\newcommand{\mb}[1]{\mathbf{#1}}

\newcommand{\mi}[1]{\mathit{#1}}

\DeclareMathOperator{\Int}{\mathit{Int}}

\DeclareMathOperator{\Outer}{\mathit{Out}}
\DeclareMathOperator{\Sing}{\mathit{Sing}}

\DeclareMathOperator{\tr}{\mathrm{tr}}

\DeclareMathOperator{\Sym}{Sym}
\DeclareMathOperator{\cyc}{\mathrm{cyc}}

\newcommand{\norm}[1]{\left\Vert#1\right\Vert}
\newcommand{\abs}[1]{\left\vert#1\right\vert}
\newcommand{\chf}[1]{\mathbf{1}_{#1}}
\newcommand{\set}[1]{\left\{#1\right\}}
\newcommand{\ip}[2]{\left \langle #1, #2 \right \rangle}
\renewcommand{\phi}{\varphi}

\newcommand{\state}[1]{\varphi \left[ #1 \right]}
\newcommand{\statep}[2]{\varphi_{#1} \left[ #2 \right]}

\newcommand{\KS}[1]{\mathit{W} \left( #1 \right)}
\newcommand{\KSP}[2]{\mathit{W}_{#1} \left( #2 \right)}
\newcommand{\Exp}[1]{\left\langle #1 \right\rangle}

\newcommand{\utimes}{\kern0.05em\buildrel{\times}\over{\rule{0em}{0.004em}} \kern-0.9em\cup \kern0.2em}
\newcommand{\putimes}{\mathop{\kern0.05em\buildrel{\times}\over{\rule{0em}{0.0em}} \kern-0.9em\cup \kern0em}}
\newcommand{\hutimes}{\mathop{\kern0.02em\buildrel{\times}\over{\rule{0em}{0.0em}} \kern-0.48em\cup \kern0em}}
\newcommand{\sutimes}{\mathrel{\kern0em \buildrel{\mathsf{x}}\over{\rule{0em}{0.0em}} \kern-0.35em\cup \kern-0.0em}}

\allowdisplaybreaks[1]

\title[Product formulas on posets, Wick products, and a correction]{Product formulas on posets, Wick products, and a correction for the $q$-Poisson process}
\author{Michael Anshelevich}
\thanks{The author was supported in part by a Simons Foundation Collaboration Grant for Mathematicians.}
\address{Department of Mathematics, Texas A\&M University, College Station, TX 77843-3368}
\email{manshel@math.tamu.edu}
\date{\today}

\begin{document}

\begin{abstract}
We give an example showing that the product and linearization formulas for the Wick product versions of the $q$-Charlier polynomials in \cite{AnsAppell} are incorrect. Next, we observe that the relation between monomials and several families of Wick polynomials is governed by ``incomplete'' versions of familiar posets. We compute M{\"o}bius functions for these posets, and prove a general poset product formula. These provide new proofs and new inversion and product formulas for Wick product versions of Hermite, Chebyshev, Charlier, free Charlier, and Laguerre polynomials. By different methods, we prove inversion formulas for the Wick product versions of the free Meixner polynomials.
\end{abstract}

\maketitle

\section{Introduction}

Let $\mc{M}$ be a complex star-algebra, and $\Exp{\cdot}$ a star-linear functional on it. Let $\Gamma(\mc{M})$ be the complex unital star-algebra generated by non-commuting symbols $\set{X(a): a \in \mc{M}}$ and $1$, subject to the linearity relations
\[
X(\alpha a + \beta b) = \alpha X(a) + \beta X(b).
\]
Thus $\Gamma(\mc{M})$ is naturally isomorphic to the tensor algebra of $\mc{M}$, but we prefer to use the polynomial notation rather than this identification. In particular the algebra is filtered by the degree of polynomials. The star-operation on $\Gamma(\mc{M})$ is determined by the requirement that $X(a^\ast) = X(a)^\ast$.

In this article we will discuss six constructions of \emph{Wick products} (four known and two new), that is, linear maps $W : \mc{M}^{\otimes n} \rightarrow \Gamma(\mc{M})$ whose ranges (together with the scalars) provide a grading compatible with the degree filtration of $\Gamma(\mc{M})$. These have also been called the Kailath-Segall polynomials. Note that in the literature, the term ``Wick product'' is also used for the multivariate Appell polynomials. See, for example, Sections~2.3 and 2.8 of \cite{AnsAppell} for differences and similarities between these two families.

 As is well-known, there are three reasons to consider Wick products.
\begin{itemize}
\item
Define a star-linear functional $\phi$ of $\Gamma(\mc{M})$ by $\state{1} = 1$, $\state{\KS{a_1 \otimes a_2 \otimes \ldots \otimes a_n}} = 0$ for $n \geq 1$. In many examples, $\phi$ is positive, the $W$ operators have orthogonal ranges for different $n$, and we have a Fock representation of $\Gamma(\mc{M})$ on (a quotient of) $L^2(\Gamma(\mc{M}), \phi)$. In this case the $W$ operators are indeed Wick products.
\item
For $\mc{M} = (L^1 \cap L^\infty)(\mf{R})$, in many examples we have an It\^{o} isometry which allows us to interpret $\KS{a_1 \otimes a_2 \otimes \ldots \otimes a_n}$ as a stochastic integral
\[
\int \ldots \int a_1(t_1) \ldots a_n(t_n) \,dX(t_1) \ldots \,dX(t_n).
\]
This isometry may involve unfamiliar inner products on multivariate function spaces, see Remark~\ref{Remark:Ito-isometry}.
\item
Suppose all $a_1 = \ldots = a_n = a$. Setting $a=1$, or more generally a multiple of a projection, $\KS{a_1 \otimes a_2 \otimes \ldots \otimes a_n}$ becomes a polynomial in $X(a)$. In many examples, these polynomials are orthogonal for different $n$.
\end{itemize}
Our main interest is in mutual expansions between polynomials $\KS{a_1 \otimes a_2 \otimes \ldots \otimes a_n}$ and monomials, product formulas
\[
\prod_{i=1}^k \KS{a_{u_i(1)} \otimes a_{u_i(2)} \otimes \ldots \otimes a_{u_i(s(i))}}
= \sum W,
\]
and corresponding linearization coefficients. We show that the product formulas for $q$-Wick products claimed in \cite{AnsAppell} are incorrect. The rest of the article concerns related positive results.

In combinatorics, linearization formulas are often proved using a weight-preserving sign-reversing involution, in other words a version of the inclusion-exclusion principle. We use a different generalization of this principle, namely M{\"o}bius inversion. For five out of our six examples, we define posets $\Pi$ such that
\[
X(a_1) \ldots X(a_n) = \sum_{\pi \in \Pi} \KS{a_1 \otimes \ldots \otimes a_n}^\pi.
\]
The posets which arise are ``incomplete'' versions of matchings, non-crossing matchings, set partitions, non-crossing partitions, and permutations. These posets, as posets, may be deserving of further study (some preliminary enumeration results for them are described in Appendix~\ref{Sec:Combinatorial}). We compute their M{\"o}bius functions and thus obtain inversion formulas. We also prove a general product formula on posets, and apply it to obtain product formulas for Wick products. For the matchings and partitions the results are known but the proofs are new. For the permutations the results are new. For the non-crossing matchings and partitions, the results are known for the usual Wick products, but the poset method allows us to extend them to operator-valued Wick products with no difficulty. As an application, it was observed by ad hoc methods that in the case of incomplete partitions, inversion formulas involve general open blocks but only singleton closed blocks (see Proposition~\ref{Prop:IP} for terminology). The M{\"o}bius function approach provides an explanation for this phenomenon. As expected, in the product formulas, in the partitions cases only inhomogeneous partitions appear, while in the permutation case we encounter ``incomplete generalized derangements''.

It is well known that the number of non-crossing matchings on $2n$ points equals the number of non-crossing partitions of $n$ points (namely it is the Catalan number).
As a combinatorial aside, we observe that the same numerical equality holds for their incomplete versions, but moreover, the collections of incomplete non-crossing matchings on $2n$ points and incomplete non-crossing partitions of $n$ points are isomorphic as posets. The reader interested only in combinatorial results may concentrate on Section \ref{Sec:Posets} and the Appendices, while the reader not interested in combinatorics may skip most of these sections.

For our sixth and most interesting example, morally corresponding to ``non-crossing permutations'', we do not know a natural poset structure governing Wick product expansions. So we perform the computations in a more direct way, using induction and generating functions. The combinatorial objects which govern these expansions are pairs $\sigma \ll \pi$ of non-crossing partitions in a relation first observed by Belinschi and Nica \cite{Belinschi-Nica-Eta,AnsFree-Meixner,Nica-NC-Linked}. In Appendix~\ref{Sec:Alternative}, we list other combinatorial structures corresponding to non-crossing permutations.

Finally, we discuss some analytic extensions of the algebraic results from earlier sections. It is natural to ask whether the map $(\mc{M}, \Exp{\cdot}) \mapsto (\Gamma(\mc{M}), \phi)$ can be interpreted as a functor, generalizing the well-known Gaussian/semiciruclar functors. For the case of the incomplete non-crossing partitions, we show that this is so, although a larger collection of morphisms would be desirable. We also observe that in all six of our examples, if the functionals $\Exp{\cdot}$ are tracial states, so are the functionals $\phi$. Finally, for the case of incomplete non-crossing partitions, we show that the product formulas hold when one of the factors is in $L^2$.

The paper is organized as follows. In Section~\ref{Sec:Posets}, we prove a general linearization result on posets, and compute M{\"o}bius functions for five posets. In Section~\ref{Sec:Wick}, we use these to obtain inversion and product formulas for five types of Wick products. In Section~\ref{Sec:Free-Meixner}, we obtain the monomial expansions and inversion formulas for the free Meixner Wick products. In short Section~\ref{Sec:Counterexample} we give an explicit counterexample to the product formula for $q$-Wick products claimed in \cite{AnsAppell}. In Section~\ref{Sec:Completions}, we collect the analytic results. In Appendix~\ref{Sec:Combinatorial}, we list some enumerative properties of incomplete posts, and combinatorial consequences of the results above for ordinary polynomials. In short Appendix~\ref{Sec:Alternative} we finish with variations on the approach in Section~\ref{Sec:Free-Meixner}.

\textbf{Acknowledgements.} I am grateful to Amudhan Krishnaswami-Usha for discussions about Section 6, and to the referee for numerous useful comments.

\section{Linearization on posets}
\label{Sec:Posets}

Throughout the article, we will only consider finite posets.

\begin{Remark}
Let $(\Pi_n, \leq)_{n=1}^\infty$ be a family of posets. As usual, we write $\sigma < \pi$ if $\sigma \leq \pi$ and $\sigma \neq \pi$. In all our examples, $\Pi_n$ will be a meet-semilattice, with the meet operation $\wedge$, and the smallest element, denoted by $\hat{0}_n$. Recall that for $\sigma \leq \pi$, the M{\"o}bius function $\mu(\sigma,\pi)$ on $\Pi$ is determined by the property that
\begin{equation}
\label{Eq:Mobius}
\sum_{\tau: \sigma \leq \tau \leq \pi} \mu(\sigma, \tau) =
\begin{cases}
1, & \sigma = \pi, \\
0, & \sigma \neq \pi.
\end{cases}
\end{equation}
As a consequence, we have the M{\"o}bius inversion formula: if $F, G$ are two functions on $\Pi$ such that
\begin{equation}
\label{Eq:F-G}
F(\pi) = \sum_{\sigma \geq \pi} G(\sigma),
\end{equation}
then
\[
G(\pi) = \sum_{\sigma \geq \pi} \mu(\pi, \sigma) F(\sigma).
\]
\end{Remark}

\begin{Thm}
\label{Thm:General-product}
Consider a family of finite posets $(\Pi_i)_{i=1}^\infty$. Fix $s(1), s(2), \ldots, s(k) \geq 1$, and denote $n = s(1) + \ldots + s(k)$. Suppose we have an order-preserving injection
\[
\alpha: \Pi_{s(1)} \times \ldots \times \Pi_{s(k)} \rightarrow \Pi_n
\]
with the property that for any $\tau \in \Pi_n$ there exists a $\tau_{s(1), \ldots, s(k)} \in \Pi_n$ with
\[
\set{\sigma \in \alpha(\Pi_{s(1)} \times \ldots \times \Pi_{s(k)}), \sigma \leq \tau} = \set{\sigma \in \Pi_n: \sigma \leq \tau_{s(1), \ldots, s(k)}},
\]
where $\tau_{s(1), \ldots, s(k)} = \tau$ if $\tau \in \alpha(\Pi_{s(1)} \times \ldots \times \Pi_{s(k)})$. Let $G$ be a function on the disjoint union $\coprod_{n=1}^\infty \Pi_n$. Denote, for $\pi \in \Pi_i$,
\[
F(\pi) = \sum_{\sigma \geq \pi} G(\sigma).
\]
Suppose that for $\pi_i \in \Pi_{s(i)}$,
\begin{equation}
\label{Eq:Multiplicative}
F(\pi_1) F(\pi_2) \ldots F(\pi_k) = F(\alpha(\pi_1, \ldots, \pi_k)).
\end{equation}
Then
\[
G(\hat{0}_{s(1)}) \ldots G(\hat{0}_{s(k)})
= \sum_{\substack{\tau \in \Pi_n \\ \tau_{s(1), \ldots, s(k)} = \hat{0}_n}} G(\tau).
\]
\end{Thm}

\begin{proof}
Taking first $\tau = \alpha(\sigma_1, \ldots, \sigma_k)$, we see that since $\tau = \tau_{s(1), \ldots, s(k)}$,
\[
\begin{split}
[\hat{0}_n, \alpha(\sigma_1, \ldots, \sigma_k)]
& = [\hat{0}_n, \tau]
= \set{\sigma \in \Pi_n: \sigma \leq \tau_{s(1), \ldots, s(k)}} \\
& = \set{\sigma \in \alpha(\Pi_{s(1)} \times \ldots \times \Pi_{s(k)}), \sigma \leq \alpha(\sigma_1, \ldots, \sigma_k)} \\
& \simeq \set{\sigma \in \Pi_{s(1)} \times \ldots \times \Pi_{s(k)}, \sigma \leq (\sigma_1, \ldots, \sigma_k)} \\
& = [\hat{0}_{s(1)}, \sigma_1] \times \ldots \times [\hat{0}_{s(k)}, \sigma_k]
\end{split}
\]
Thus, since the M{\"o}bius function is multiplicative,
\[
\mu(\hat{0}_n, \alpha(\sigma_1, \ldots, \sigma_k))
= \prod_{i=1}^k \mu(\hat{0}_{s(i)}, \sigma_i).
\]
The rest of the proof is similar to that of Theorem~4 in \cite{Rota}. Using various assumptions,
\[
\begin{split}
G(\hat{0}_{s(1)}) \ldots G(\hat{0}_{s(k)})
& = \sum_{\sigma_1 \in \Pi_{s(1)}} \ldots \sum_{\sigma_k \in \Pi_{s(k)}} \prod_{i=1}^k \mu(\hat{0}_{s(i)}, \sigma_i) \prod_{i=1}^k F(\sigma_i) \\
& = \sum_{(\sigma_1, \ldots, \sigma_k) \in \Pi_{s(1)} \times \ldots \times \Pi_{s(k)}} \mu(\hat{0}_n, \alpha(\sigma_1, \ldots, \sigma_k)) F(\alpha(\sigma_1, \ldots, \sigma_k)) \\
& = \sum_{(\sigma_1, \ldots, \sigma_k) \in \Pi_{s(1)} \times \ldots \times \Pi_{s(k)}} \mu(\hat{0}_n, \alpha(\sigma_1, \ldots, \sigma_k)) \sum_{\tau \geq \alpha(\sigma_1, \ldots, \sigma_k)} G(\tau) \\
& = \sum_{\tau \in \Pi_n} G(\tau) \sum_{\substack{(\sigma_1, \ldots, \sigma_k) \in \Pi_{s(1)} \times \ldots \times \Pi_{s(k)}, \\ \alpha(\sigma_1, \ldots, \sigma_k) \leq \tau}} \mu(\hat{0}_n, \alpha(\sigma_1, \ldots, \sigma_k)) \\
& = \sum_{\tau \in \Pi_n} G(\tau) \sum_{\substack{\sigma \in \Pi_n, \\ \sigma \leq \tau_{s(1), \ldots, s(k)}}} \mu(\hat{0}_n, \sigma) \\
& = \sum_{\substack{\tau \in \Pi_n \\ \tau_{s(1), \ldots, s(k)} = \hat{0}_n}} G(\tau). \qedhere
\end{split}
\]
\end{proof}

\begin{Remark}
We will show that for the five posets in the next series of propositions, the conditions above are satisfied, and compute the corresponding $\tau_{s(1), \ldots, s(k)}$. In all the examples, $\alpha$ combines objects defined on each of the subintervals
\begin{equation}
\label{Eq:Intervals}
J_1 = [1, \ldots, s(1)], \quad J_2 = [s(1) + 1, \ldots, s(1) + s(2)], \quad \ldots, \quad J_k = [n - s(k) + 1, \ldots, n]
\end{equation}
into a single object on the interval $[1, \ldots, n]$, in a natural way. We will denote by $(\hat{1}_{s(1)}, \ldots, \hat{1}_{s(k)})$ the partition on $[n]$ whose blocks are these intervals. Note that $\Pi_i$ is \emph{not} assumed to have a maximal element, so $\hat{1}_i$ only denotes the maximal element of $\mc{P}(i)$.
\end{Remark}

\begin{Notation}[Background on partitions]
Denote $\Sing(\pi)$ the single-element blocks of a partition $\pi$, and $\mi{Pair}(\pi)$ the two-element blocks.

A partition $\pi$ of an ordered set $\Lambda$ is \emph{non-crossing} if there are no two blocks $U \neq V$ of $\pi$ with $i, k \in U$, $j, l \in V$, and $i < j < k < l$. The set of non-crossing partitions is denotes by $\mc{NC}(\Lambda)$, or $\mc{NC}(n)$ in case $\Lambda = [n] = \set{1, 2, \ldots, n}$. A block $U \in \pi \in \mc{NC}(n)$ containing $i_0$ has \emph{depth $k$} if $k$ is the largest integer (starting with $0$) for which there is exist $i_k < i_{k-1} < \ldots < i_1 < i_0 < j_1 < \ldots < j_k$ such that all $i_0, i_1, \ldots, i_k$ belong to different blocks of $\pi$ while $i_u \stackrel{\pi}{\sim} j_u$, $u = 1, \ldots, k$. A block is \emph{outer} if it has depth zero, and \emph{inner} otherwise. Denote $\Outer(\pi)$ the outer blocks of $\pi$. Denote $\mc{NC}_{\geq 2}(\Lambda)$ the non-crossing partitions with no singletons. The \emph{interval partitions} $\Int(\Lambda)$ are partitions whose blocks are intervals.
\end{Notation}

\begin{Prop}
\label{Prop:P12}
Denote by $\mc{P}_{1,2}(n)$, the \emph{incomplete matchings}, the partitions of $[n]$ into pairs and singletons. Equip it with the poset structure it inherits from the usual refinement order on partitions. Then the M{\"o}bius function on this poset is
\[
\mu(\pi, \sigma) = (-1)^{\abs{\mi{Pair}(\sigma)} - \abs{\mi{Pair}(\pi)}}.
\]
Also for this poset, $\tau_{s(1), \ldots, s(k)} = \tau \wedge (\hat{1}_{s(1)}, \ldots, \hat{1}_{s(k)})$.
\end{Prop}

\begin{proof}
It suffices to note that, denoting by $U$ the pairs and $V$ the singleton blocks,
\[
[(U_1, \ldots, U_u, V_1', \ldots, V_v'), (U_1, \ldots, U_k, V_1, \ldots, V_\ell)] \simeq [\hat{0}, (U_{u+1}, \ldots, U_k)] \simeq \mc{P}(2)^{k-u}. \qedhere
\]
\end{proof}

\begin{Prop}
\label{Prop:NC12}
Denote by $\mc{INC}_{1,2}(n)$, the \emph{incomplete non-crossing matchings}, the non-crossing partitions of $[n]$ into pairs and singletons, such that all singletons are outer. Equip it with the poset structure inherited from $\mc{P}_{1,2}(n)$. Then the M{\"o}bius function on this poset is
\begin{multline*}
\mu((U_1, \ldots, U_u, V_1', \ldots, V_v'), (U_1, \ldots, U_k, V_1, \ldots, V_\ell)) \\
=
\begin{cases}
(-1)^{k-u}, & \forall (u+1) \leq i \leq k: U_i \in \Outer(\pi), \\
0, & \text{ otherwise}.
\end{cases}
\end{multline*}
For this poset, $\tau_{s(1), \ldots, s(k)}$ is given by the same expression as for $\mc{P}_{1,2}$.
\end{Prop}

\begin{proof}
Clearly
\[
[(U_1, \ldots, U_u, V_1', \ldots, V_v'), (U_1, \ldots, U_k, V_1, \ldots, V_\ell)] \simeq [\hat{0}, (U_{u+1}, \ldots, U_k)].
\]
Moreover this interval is the product of intervals from $\hat{0}$ to a partition with a single outer block. If that partition is simply $\set{U}$, the M{\"o}bius function $\mu(\hat{0}_2, \set{U}) = (-1)$. On the other hand, if it is a larger partition with the single outer block $\set{i,j}$, recalling that all singletons in $\mc{INC}_{1,2}$ are outer, we see that
\[
\sigma < (U_1', U_2', \ldots, \set{i,j}) \quad \Leftrightarrow \sigma \leq (U_1', \ldots, U_{s-1}', \set{i}, \set{j}).
\]
and so from property \eqref{Eq:Mobius},
\[
\mu(\hat{0}, (U_1', U_2', \ldots, \set{i,j})) = \sum_{\sigma \leq (U_1', U_2', \ldots, \set{i,j})} \mu(\hat{0}, \sigma) - \sum_{\sigma < (U_1', U_2', \ldots, \set{i,j})} \mu(\hat{0}, \sigma) = 0. \qedhere
\]
\end{proof}

\begin{Defn}
Denote by $\mc{IP}(n)$, the \emph{incomplete partitions}, the collection
\[
\set{(\pi, S): \pi \in \mc{P}(n), S \subset \pi}.
\]
Here, and in the subsequent examples, the elements of $S$ will be called \emph{open blocks}, those of $\pi \setminus S$ \emph{closed blocks}. Denote $\bigcup S = \bigcup_{V \in S} V$ the union of all the open blocks. Equip $\mc{IP}(n)$ with the poset structure
\[
(\pi, S) \leq (\sigma, T) \text{ if } U \in \pi \setminus S \Rightarrow U \in \sigma \setminus T \text{ and } \pi|_{\bigcup S} \leq \sigma|_{\bigcup S}.
\]
\end{Defn}

\begin{Prop}
\label{Prop:IP}
The M{\"o}bius function on $\mc{IP}(n)$ is
\[
\mu((\hat{0}, \hat{0}), (\pi, S)) =
\begin{cases}
(-1)^{n - \abs{S}} \prod_{V \in S} (\abs{V} - 1)!, & \forall U \in \pi \setminus S: \abs{U} = 1, \\
0, & \text{ otherwise}.
\end{cases}
\]
Also for this poset, $(\tau, S)_{s(1), \ldots, s(k)} = (\tau \wedge (\hat{1}_{s(1)}, \ldots, \hat{1}_{s(k)}), T)$, where $U \in (\tau \wedge (\hat{1}_{s(1)}, \ldots, \hat{1}_{s(k)})) \setminus T$ if and only if $U \in \tau \setminus S$. In particular, $(\tau, S)_{s(1), \ldots, s(k)} = (\hat{0}_n, \hat{0}_n)$ if and only if $\tau \wedge (\hat{1}_{s(1)}, \ldots, \hat{1}_{s(k)}) = \hat{0}_n$ and all singletons of $\tau$ are open.
\end{Prop}

\begin{proof}
Clearly
\[
\begin{split}
[(\hat{0}, \hat{0}), (\pi, S)]
\simeq \prod_{U \in \pi \setminus S} [(\hat{0}, \hat{0}), (\set{U}, \emptyset)] \times \prod_{V \in S} [(\hat{0}, \hat{0}), (\set{V}, \set{V})]
\end{split}
\]
and $[(\hat{0}, \hat{0}), (\set{V}, \set{V})] \simeq [\hat{0}, \set{V}] \simeq \mc{P}(\abs{V})$, so $\mu((\hat{0}, \hat{0}), (\set{V}, \set{V})) = (-1)^{\abs{V} - 1} (\abs{V} - 1)!$. Also
\[
\set{(\sigma, T) < (\set{U}, \emptyset)} = \set{(\sigma, T) \leq (\set{U}, \set{U})}
\]
so using property~\eqref{Eq:Mobius}, $\mu((\hat{0}, \hat{0}), (\set{U}, \emptyset)) = 0$ unless $\abs{U} = 1$. The formula for the M{\"o}bius function follows. The final formula follows from the definition of the order.
\end{proof}

\begin{Prop}
\label{Prop:INC}
Denote by $\mc{INC}(n)$, the \emph{incomplete non-crossing partitions} (called the linear non-crossing half-permutations in \cite{Kusalik-Mingo-Speicher}), the collection
\[
\set{(\pi, S): \pi \in \mc{NC}(n), S \subset \Outer(\pi)}.
\]
Equip it with the poset structure inherited from $\mc{IP}(n)$. Then the posets $\mc{INC}(n)$ and $\mc{INC}_{1,2}(2n)$ are isomorphic. Under this isomorphism, partitions with $\ell$ open blocks are mapped to partitions with $2 \ell$ singletons, and partitions with $k$ closed blocks are mapped to partitions with $k$ pairs at even depth. In particular, the M{\"o}bius function on this poset is
\begin{equation}
\label{Eq:Mobius-INC}
\mu((\hat{0}, \hat{0}), (\pi, S)) =
\begin{cases}
(-1)^{n - \abs{S}}, & \pi \in \Int(n), \forall U \in \pi \setminus S: \abs{U} = 1, \\
0, & \text{ otherwise}.
\end{cases}
\end{equation}
For this poset, $\tau_{s(1), \ldots, s(k)}$ is given by the same expression as for $\mc{IP}$.
\end{Prop}

\begin{proof}
We exhibit an bijection between $\mc{INC}(n)$ and $\mc{INC}_{1,2}(\set{1, \bar{1}, \ldots, n, \bar{n}})$. A closed block $(i_1 < i_2 < \ldots < i_k)$ is replaced by blocks $(i_1, \bar{i}_k), (\bar{i}_1, i_2), \ldots, (\bar{i}_{k-1}, i_k)$, while the open block with these elements is replaced by $(i_1), (\bar{i}_k), (\bar{i}_1, i_2), \ldots, (\bar{i}_{k-1}, i_k)$. This is easily seen to be a bijection. Moreover, combining two (outer) open blocks $(i_1 < i_2 < \ldots < i_k)$ and $(j_1 < j_2 < \ldots < j_\ell)$ with $i_k < j_1$ corresponds to pairing off $\bar{i}_k$ and $j_1$, while closing the open block $(i_1 < i_2 < \ldots < i_k)$ corresponds to pairing off $i_1$ and $\bar{i}_k$. So this map is a poset isomorphism. Finally, each open block of $\pi \in \mc{INC}(n)$ produces two singletons in the image. The statement about closed blocks follows from a recursive argument.

Next, for $\sigma \in \mc{INC}_{1,2}(2n)$, from Proposition~\ref{Prop:NC12}
\[
\mu_{\mc{INC}_{1,2}}(\hat{0}, \sigma) =
\begin{cases}
(-1)^{\abs{\mi{Pair}(\sigma)}}, & \sigma \in \Int(2n), \\
0, \text{otherwise}.
\end{cases}
\]
Clearly the bijection above maps incomplete interval matchings onto the set of incomplete interval partitions all of whose closed blocks are singletons. It remains to note that if $\sigma$ is mapped to $(\pi, S)$, then $\abs{\mi{Pair}(\sigma)} = \frac{1}{2} (2 n - \abs{\Sing(\sigma)}) = n - \abs{S}$. The formula for $\tau_{s(1), \ldots, s(k)}$ follows from the definition of the order.
\end{proof}

\begin{Defn}
Denote by $\mc{IPRM}(n)$, the \emph{incomplete permutations} (sometimes called partial permutations \cite{Borwein-Injective-partial}, although this term has also been used for different objects), the collection of maps
\[
\mc{IPRM}(n) = \set{(\Lambda, f): \Lambda \subset [n], f : \Lambda \rightarrow [n] \text{ injective}}.
\]
Equip $\mc{IPRM}(n)$ with the following poset structure:
\[
(\Lambda, f) \leq (\Omega, g) \text{ if } \Lambda \subset \Omega \text{ and } g|_\Lambda = f.
\]
\end{Defn}

\begin{Prop}
$\mc{IPRM}(n)$ is isomorphic as a poset to pairs $(\pi, S)$, where $\pi$ is a partition of $[n]$ with an order on each block of the partition, $S$ is a collection of some blocks of this partition, and the order on the blocks in $\pi \setminus S$ is defined only up to a cyclic permutation. Equivalently, these are collections of words in $[n]$, where each letter appears exactly once, and some of the words are defined only up to cyclic order. In the poset structure, $(\pi, S) \leq (\sigma, T)$  if $U \in \pi \setminus S \Rightarrow U \in \sigma \setminus T$, the restriction of partitions $\pi|_{\bigcup S} \leq \sigma|_{\bigcup S}$, and the words corresponding to blocks of $\sigma$ combined out of blocks of $\pi$ are obtained by concatenating the words corresponding to these blocks of $\pi$, in some order (and the combined word possibly cyclically rotated if the block of $\sigma$ is closed).
\end{Prop}

\begin{proof}
Given $(\Lambda, f)$ as in the definition of $\mc{IPRM}(n)$, we define the partition $\pi'$ of $\Lambda \cup f(\Lambda)$ to be the partition into orbits of $f$, that is, largest subsets $\set{w_1, \ldots, w_\ell}$ such that $w_{i+1} = f(w_i)$. The block is in $\pi' \setminus S$ if $w_\ell \in \Lambda$, so that $f(w_{\ell}) = w_1$, and the block is in $S$ if $w_\ell \in f(\Lambda) \setminus \Lambda$. Complete $\pi'$ to a partition $\pi$ of $[n]$ by letting each elements in $[n] \setminus (\Lambda \cup f(\Lambda))$ be a singleton block in $S$. Note that each block in $S$ has an order structure $w_1 < w_2 < \ldots < w_\ell$, and each block in $\pi \setminus S$ has a cyclic order. Conversely, we can recover $(\Lambda, f)$ from $(\pi, S)$ by setting $[n] \setminus \Lambda$ to consist of the largest elements (according to the block order) of blocks in $S$, and $f$ defined by the order on the blocks.

Next, let $(\Lambda, f) \leftrightarrow (\pi, S)$, $(\Omega, g) \leftrightarrow (\sigma, T)$, and $(\Lambda, f) \leq (\Omega, g)$. The orbits of $g$, ordered according to the mapping structure of $g$, have the form either
\[
\set{w_1 < w_2 < \ldots < w_\ell} \in \pi \setminus S,
\]
or
\[
\set{v_{1, 0} < \ldots < v_{k(0), 0} < w_{1, 1} < \ldots < w_{\ell(1), 1} < v_{1, 1} < \ldots < v_{k(1), 1} < w_{1, 2} < \ldots < w_{\ell(2), 2} < \ldots },
\]
where all $v_{i, j} \in (\Omega \cup g(\Omega)) \setminus (\Lambda \cup f(\Lambda))$ are singletons in $S$ and each $\set{w_{1, j} < \ldots w_{\ell(j), j}} \in S$ is an orbit of $f$ in $S$. The description above follows.
\end{proof}

\begin{Example}
For $\Lambda = \set{2, 3, 5, 7, 8, 9} \subset [9]$ and
\[
f(2) = 5, f(3) = 8, f(5) = 4, f(7) = 3, f(8) = 7, f(9) = 9,
\]
the corresponding incomplete partition with ordered blocks is
\[
\pi = \set{(1), (2 < 5 < 4), (3 < 8 < 7 < 3), (6), (9)}, \quad S = \set{(1), (2 < 5 < 4), (6)}.
\]
\end{Example}

\begin{Prop}
\label{Prop:IPRM}
The M{\"o}bius function on $\mc{IPRM}(n)$ is
\[
\mu((\hat{0}, \hat{0}), (\pi, S)) = (-1)^{n - \abs{S}}.
\]
$(\Lambda, f)_{s(1), \ldots, s(k)} = (\Omega, g)$, where
\[
\Omega = \bigcup_{i=1}^k \set{x \in \Lambda \cap J_i: f(x) \in J_i}
\]
and $g = f|_\Omega$. In particular, $(\Lambda, f)_{s(1), \ldots, s(k)}$ equals the minimal element of $\mc{IPRM}(n)$ if for each $i$ and $x \in \Lambda \cap J_i$, $f(x) \not \in J_i$. We will call this final family \emph{incomplete derangements} and denote it by $\mc{ID}(s(1), \ldots, s(k))$.
\end{Prop}

\begin{proof}
The blocks of $\pi$ are simply the orbits of $f$, with elements of $[n] \setminus \Lambda$ included as open singletons. To compute the M{\"o}bius function, it suffices to assume that $f$ has a single orbit. Elements smaller than $(\Lambda, f)$ are in an ordered bijection with subsets of $\Lambda$, and the M{\"o}bius function $(-1)^{\abs{\Lambda}}$. It remains to note that $\Lambda = [n]$ if the corresponding block is closed, and $\abs{\Lambda} = n-1$ if the corresponding block is open. Formulas for $(\Lambda, f)_{s(1), \ldots, s(k)}$ follow from the definition of the order.
\end{proof}

\section{Multiplication of Wick products}
\label{Sec:Wick}

Let $\mc{M}, \Gamma(\mc{M})$ be as in the beginning of the introduction. 

\begin{Notation}
Order the blocks of a partition according to the order of the largest elements of the blocks. For an ordered index set $\Lambda$, denote $a_\Lambda = \prod_{i \in \Lambda} a_i$. Finally, write $J_i = \set{u_i(1), \ldots, u_i(s(i))}$, so that
\[
\set{1, 2, \ldots, n} = (u_1(1), \ldots, u_1(s(1)), u_2(1), \ldots, u_2(s(2)), \ldots, u_k(1), \ldots, u_k(s(k))).
\]
\end{Notation}

\begin{Prop}
\label{Prop:classical-expansions-1}
Define $\KSP{\mc{P}_{1,2}}{a_1 \otimes a_2 \otimes \ldots \otimes a_n}$ recursively by
\[
\begin{split}
\KSP{\mc{P}_{1,2}}{a_1 \otimes \ldots \otimes a_n \otimes a_{n+1}}
& = \KSP{\mc{P}_{1,2}}{a_1 \otimes \ldots \otimes a_n} X(a_{n+1}) \\
&\quad - \sum_{i=1}^n \KSP{\mc{P}_{1,2}}{a_1 \otimes \ldots \otimes \hat{a}_i \otimes \ldots \otimes a_n} \Exp{a_i a_{n+1}}.
\end{split}
\]
Then
\begin{equation}
\label{Eq:X-P12}
X(a_1) \ldots X(a_n) = \sum_{\pi \in \mc{P}_{1,2}(n)} \prod_{U \in \pi: \abs{U} = 2} \Exp{a_U} \KSP{\mc{P}_{1,2}}{\bigotimes_{V \in \pi: \abs{V} = 1} a_V} ,
\end{equation}
and
\[
\KSP{\mc{P}_{1,2}}{a_1 \otimes \ldots \otimes a_n} = \sum_{\pi \in \mc{P}_{1,2}(n)} (-1)^{\abs{\pi} - \abs{\Sing(\pi)}} \prod_{U \in \pi: \abs{U} = 2} \Exp{a_U} \prod_{V \in \pi: \abs{V} = 1} X(a_V),
\]
and
\begin{multline*}
\prod_{i=1}^k \KSP{\mc{P}_{1,2}}{a_{u_i(1)} \otimes a_{u_i(2)} \otimes \ldots \otimes a_{u_i(s(i))}} \\
= \sum_{\substack{\pi \in \mc{P}_{1,2}(n) \\ \pi \wedge (\hat{1}_{s(1)}, \ldots, \hat{1}_{s(k)}) = \hat{0}_n}} \prod_{U \in \pi: \abs{U} = 2} \Exp{a_U} \KSP{\mc{P}_{1,2}}{\bigotimes_{V \in \pi: \abs{V} = 1} a_V} .
\end{multline*}
\end{Prop}

\begin{proof}
Equation~\eqref{Eq:X-P12} is well known, see for example Theorem~2.1 in \cite{Effros-Popa} for $q=1$. It implies that for $\pi \in \mc{P}_{1,2}(n)$,
\[
\prod_{U \in \pi: \abs{U} = 2} \Exp{a_U} \prod_{V \in \pi: \abs{V} = 1} X(a_V) = \sum_{\substack{\sigma \in \mc{P}_{1,2}(n) \\ \sigma \geq \pi}} \prod_{U \in \sigma: \abs{U} = 2} \Exp{a_U} \KSP{\mc{P}_{1,2}}{\bigotimes_{V \in \sigma: \abs{V} = 1} a_V}.
\]
Denoting the left-hand-side of this equation by $F(\pi)$ and each term in the sum on the right-hand-side by $G(\sigma)$, we see that these functions satisfy the relation \eqref{Eq:F-G}, and $F$ satisfies the multiplicative property \eqref{Eq:Multiplicative}. So Theorem~\ref{Thm:General-product} and Proposition~\ref{Prop:P12} imply the results.
\end{proof}

\begin{Prop}
\label{Prop:classical-expansions-2}
Define $\KSP{\mc{IP}}{a_1 \otimes a_2 \otimes \ldots \otimes a_n}$ recursively by
\[
\begin{split}
\KSP{\mc{IP}}{a_1 \otimes \ldots \otimes a_n \otimes a_{n+1}}
& = \KSP{\mc{IP}}{a_1 \otimes \ldots \otimes a_n} X(a_{n+1}) - \KSP{\mc{IP}}{a_1 \otimes \ldots \otimes a_n} \Exp{a_{n+1}} \\
&\quad - \sum_{i=1}^n \KSP{\mc{IP}}{a_1 \otimes \ldots \otimes \hat{a}_i \otimes \ldots \otimes a_{n} \otimes a_i a_{n+1}} \\
&\quad - \sum_{i=1}^n \KSP{\mc{IP}}{a_1 \otimes \ldots \otimes \hat{a}_i \otimes \ldots \otimes a_n} \Exp{a_i a_{n+1}}.
\end{split}
\]
Then
\begin{equation}
\label{Eq:X-IP}
X(a_1) \ldots X(a_n) = \sum_{(\pi, S) \in \mc{IP}(n)} \prod_{U \in \pi \setminus S} \Exp{a_U} \KSP{\mc{IP}}{\bigotimes_{V \in S} a_V}.
\end{equation}
If $\mc{M}$ is commutative, then
\[
\KSP{\mc{IP}}{a_1 \otimes \ldots \otimes a_n} = \sum_{\substack{(\pi, S) \in \mc{IP}(n) \\ U \in \pi \setminus S \Rightarrow \abs{U} = 1}} (-1)^{n - \abs{S}} \prod_{U \in \pi \setminus S} \Exp{a_U} \prod_{V \in S} (\abs{V} - 1)! X(a_V),
\]
and
\begin{multline*}
\prod_{i=1}^k \KSP{\mc{IP}}{a_{u_i(1)} \otimes a_{u_i(2)} \otimes \ldots \otimes a_{u_i(s(i))}}
= \sum_{\substack{(\pi, S) \in \mc{IP}(n) \\ \pi \wedge (\hat{1}_{s(1)}, \ldots, \hat{1}_{s(k)}) = \hat{0}_n \\ \Sing(\pi) \subset S}} \prod_{U \in \pi \setminus S} \Exp{a_U} \KSP{\mc{IP}}{\bigotimes_{V \in S} a_V}.
\end{multline*}
\end{Prop}

\begin{proof}
Equation~\eqref{Eq:X-IP} is known, see for example Proposition~2.7(a) in \cite{AnsAppell}. For commutative $\mc{M}$, it implies that for $(\pi, S) \in \mc{IP}(n)$,
\[
\prod_{U \in \pi \setminus S} \Exp{a_U} \prod_{V \in S} X(a_V) = \sum_{\substack{(\sigma, T) \in \mc{IP}(n) \\ (\sigma, T) \geq (\pi, S)}} \prod_{U \in \sigma \setminus T} \Exp{a_U} \KSP{\mc{IP}}{\bigotimes_{V \in T} a_V}.
\]
So Theorem~\ref{Thm:General-product} and Proposition~\ref{Prop:IP} imply the results.
\end{proof}

\begin{Thm} [Cf. Section~4 in \cite{Sniady-SWN}]
\label{Thm:5th-expansion}
Define $\KSP{\mc{IPRM}}{a_1 \otimes a_2 \otimes \ldots \otimes a_n}$ recursively by
\[
\begin{split}
& \KSP{\mc{IPRM}}{a_1 \otimes \ldots \otimes a_n \otimes a_{n+1}} \\
&\quad = \KSP{\mc{IPRM}}{a_1 \otimes \ldots \otimes a_n} X(a_{n+1}) - \KSP{\mc{IPRM}}{a_1 \otimes \ldots \otimes a_n} \Exp{a_{n+1}} \\
&\quad - \sum_{i=1}^n \KSP{\mc{IPRM}}{a_1 \otimes \ldots \otimes \hat{a}_i \otimes \ldots \otimes a_{n} \otimes a_i a_{n+1}} \\
&\quad - \sum_{i=1}^n \KSP{\mc{IPRM}}{a_1 \otimes \ldots \otimes \hat{a}_i \otimes \ldots \otimes a_n \otimes a_{n+1} a_i} \\
&\quad - \sum_{i=1}^n \KSP{\mc{IPRM}}{a_1 \otimes \ldots \otimes \hat{a}_i \otimes \ldots \otimes a_n} \Exp{a_i a_{n+1}} \\
&\quad - \sum_{1 \leq i < j \leq n} \KSP{\mc{IPRM}}{a_1 \otimes \ldots \otimes \hat{a}_i \otimes \ldots \otimes \hat{a}_j \otimes \ldots \otimes a_{n} \otimes a_i a_{n+1} a_j} \\
&\quad - \sum_{1 \leq i < j \leq n} \KSP{\mc{IPRM}}{a_1 \otimes \ldots \otimes \hat{a}_i \otimes \ldots \otimes \hat{a}_j \otimes \ldots \otimes a_{n} \otimes a_j a_{n+1} a_i}.
\end{split}
\]
For $(\Lambda, f) \in \mc{IPRM}(n)$, let $(\pi, S) \in \mc{INC}(n)$ be the corresponding orbit decomposition. Order each block $\set{w_1, w_2, \ldots, w_\ell}$ of $\pi$ so that $w_{i+1} = f(w_i)$ and, in case the block is closed, so that $w_\ell$ is the numerically largest element in the block. Denote
\[
\KSP{\mc{IPRM}}{a_1 \otimes \ldots \otimes a_n}^{(\Lambda, f)}
= \prod_{U \in \pi \setminus S} \Exp{a_U} \KSP{\mc{IPRM}}{\bigotimes_{V \in S} a_V}
\]
and
\[
M(a_1 \otimes \ldots \otimes a_n)^{(\Lambda, f)} = \prod_{U \in \pi \setminus S} \Exp{a_U} \prod_{V \in S} X(a_V),
\]
where on each block of $\pi$ we use the order described above. Then
\begin{equation}
\label{Eq:X-IPRM}
X(a_1) \ldots X(a_n) = \sum_{(\Lambda, f) \in \mc{IPRM}(n)} \KSP{\mc{IPRM}}{a_1 \otimes \ldots \otimes a_n}^{(\Lambda, f)}.
\end{equation}
Also,
\[
\KSP{\mc{IPRM}}{a_1 \otimes \ldots \otimes a_n} = \sum_{(\Lambda, f) \in \mc{IPRM}(n)} (-1)^{n - \abs{S}} M(a_1 \otimes \ldots \otimes a_n)^{(\Lambda, f)},
\]
and
\begin{multline*}
\prod_{i=1}^k \KSP{\mc{IPRM}}{a_{u_i(1)} \otimes a_{u_i(2)} \otimes \ldots \otimes a_{u_i(s(i))}} \\
= \sum_{(\Lambda, f) \in \mc{ID}(s(1), \ldots, s(k))} \KSP{\mc{IPRM}}{a_1 \otimes \ldots \otimes a_n}^{(\Lambda, f)}.
\end{multline*}
\end{Thm}

\begin{proof}
We prove equation~\eqref{Eq:X-IPRM} by induction. $X(a) = \KSP{\mc{IPRM}}{a} + \Exp{a}$. Denoting
\[
S = \set{V_1 < V_2 < \ldots < V_{\abs{S}}}
\]
and using the inductive hypothesis,
\begin{align*}
& X(a_1) \ldots X(a_n) X(a_{n+1})
= \sum_{(\Lambda, f) \in \mc{IPRM}(n)} \KSP{\mc{IPRM}}{a_1 \otimes \ldots \otimes a_n}^{(\Lambda, f)} X(a_{n+1}) \\
& = \sum_{(\Lambda, f) \in \mc{IPRM}(n)} \prod_{U \in \pi \setminus S} \Exp{a_U} \\
&\qquad \Biggl(\KS{a_{V_1} \otimes a_{V_2} \otimes \ldots \otimes a_{V_{\abs{S}}} \otimes a_{n+1}} \\
&\qquad\quad + \KS{a_{V_1}\otimes a_{V_2}\otimes \ldots a_{V_{\abs{S}}}} \Exp{a_{n+1}}\\
&\qquad\quad + \chf{\abs{S} \geq 1} \sum_{i=1}^{\abs{S}} \KS{a_{V_1} \otimes a_{V_2} \otimes \ldots \otimes \hat{a}_{V_i}\otimes \ldots\otimes a_{V_{\abs{S}}}\otimes a_{V_i} a_{n+1}} \\
&\qquad\quad + \chf{\abs{S} \geq 1} \sum_{i=1}^{\abs{S}} \KS{a_{V_1}\otimes a_{V_2} \otimes \ldots \otimes \hat{a}_{V_i} \otimes \ldots \otimes a_{V_{\abs{S}}} \otimes a_{n+1} a_{V_i}}\\
&\qquad\quad + \chf{\abs{S} \geq 1} \sum_{i=1}^{\abs{S}} \KS{a_{V_1} \otimes a_{V_2} \otimes \ldots \otimes \hat{a}_{V_i} \otimes \ldots \otimes a_{V_{{\abs{S}}}}} \Exp{a_{V_i} a_{n+1}} \\
&\qquad\quad + \chf{\abs{S} \geq 2} \sum_{1 \leq i < j \leq \abs{S}} \KS{a_{V_1} \otimes \ldots \otimes \hat{a}_{V_i} \otimes \ldots \otimes \hat{a}_{V_j} \otimes \ldots \otimes a_{V_{\abs{S}}} \otimes a_{V_i} a_{n+1} a_{V_j}} \\
&\qquad\quad + \chf{\abs{S} \geq 2} \sum_{1 \leq i < j \leq \abs{S}} \KS{a_{V_1} \otimes \ldots \otimes \hat{a}_{V_i} \otimes \ldots \otimes \hat{a}_{V_j} \otimes \ldots \otimes a_{V_{\abs{S}}} \otimes a_{V_j} a_{n+1} a_{V_i}} \Biggr).
\end{align*}
The first term produces all the partitions in $\mc{IPRM}(n+1)$ in which $n+1$ is an open singleton. The second term produces all the partitions in which $n+1$ is a closed singleton. The third term produces all the partitions in which $n+1$ is a final letter in an open word of length at least $2$. The fourth term produces all the partitions in which $n+1$ is the initial letter in an open word of length at least $2$. The fifth term produces all the partitions in which $n+1$ is contained in a closed word of length at least $2$. The sixth term produces all the partitions in which $n+1$ is contained in an open word of length at least $3$, is neither the initial nor the final letter in it, and the largest letter preceding it is smaller than the largest letter following it. The seventh term produces all the partitions in which $n+1$ is contained in an open word of length at least $3$, is neither the initial nor the final letter in it, and the largest letter preceding it is larger than the largest letter following it. These seven classes are disjoint and exhaust $\mc{IPRM}(n+1)$.

It follows that for $(\Lambda, f) \in \mc{IPRM}(n)$,
\[
M(a_1 \otimes \ldots \otimes a_n)^{(\Lambda, f)} =
\sum_{\substack{(\Omega, g) \in \mc{IPRM}(n) \\ (\Omega, g) \geq (\Lambda, f)}} \KSP{\mc{IPRM}}{a_1 \otimes \ldots \otimes a_n}^{(\Omega, g)}.
\]
So Theorem~\ref{Thm:General-product} and Proposition~\ref{Prop:IPRM} imply the results.
\end{proof}

\begin{Remark}
Assume additionally that $\Exp{\cdot}$ is a trace. By applying the functional $\phi_{\mc{IPRM}}$ to \eqref{Eq:X-IPRM}, we obtain the moment formula for $\set{X(a_i)}$, which implies that the cumulants of $\phi_{\mc{IPRM}}$ are
\[
K^{\phi_{\mc{IPRM}}}[X(a_1), \ldots, X(a_n)] = \frac{1}{n} \sum_{\alpha \in \Sym(n)} \Exp{a_{\alpha(1)} \ldots a_{\alpha(n)}}.
\]
\end{Remark}

\begin{Remark}

Note that in Propositions~\ref{Prop:classical-expansions-1} and \ref{Prop:classical-expansions-2}, we do not assume that $W$ is symmetric in its arguments, and in Theorem~\ref{Thm:5th-expansion}, we do not assume that $\mc{M}$ is commutative. The results in Proposition~\ref{Prop:classical-expansions-1} are known by direct methods, see Theorems~3.1 and 3.3 in \cite{Effros-Popa}. The results in Proposition~\ref{Prop:classical-expansions-2} are stated in Proposition~2.7 of \cite{AnsAppell}.

If $\mc{M}$ is commutative, $\KSP{\mc{IPRM}}{a_1 \otimes \ldots \otimes a_n}$ depends only on the underlying incomplete partition, so we may re-write the expansions in Theorem~\ref{Thm:5th-expansion} as
\begin{equation*}
X(a_1) \ldots X(a_n) = \sum_{(\pi, S) \in \mc{IP}(n)} \prod_{U \in \pi \setminus S} (\abs{U} - 1)! \Exp{a_U} \prod_{V \in S} (\abs{V})! \KSP{\mc{IPRM}}{\bigotimes_{V \in S} a_V},
\end{equation*}
and
\[
\KSP{\mc{IPRM}}{a_1 \otimes \ldots \otimes a_n} = \sum_{(\pi, S) \in \mc{IP}(n)} (-1)^{n - \abs{S}} \prod_{U \in \pi \setminus S} (\abs{U} - 1)! \Exp{a_U} \prod_{V \in S} (\abs{V})! X(a_V),
\]
Note that this additional assumption does not imply that $\Gamma(\mc{M})$ is commutative; it is however natural to assume such commutativity to have a non-degenerate representation, see Section~\ref{Sec:Completions}.
\end{Remark}

\begin{Remark}
Let $(\Lambda, f) \in \mc{IPRM}(n)$. For $w \in [n]$, we say that it is
\begin{itemize}
\item
A valley if $w \not \in \Lambda \cup f(\Lambda)$, or $w \in \Lambda \setminus f(\Lambda)$ and $w < f(w)$, or $w \in f(\Lambda) \setminus \Lambda$ and $f^{-1}(w) > w$, or $w \in \Lambda \cup f(\Lambda)$ and $f^{-1}(w) > w < f(w)$.
\item
A closed singleton if $f(w) = w$.
\item
A double rise if $f^{-1}(w) > w$ and either $w > f(w)$ or $w \not \in \Lambda$.
\item
A double fall if $w < f(w)$ and either $f^{-1}(w) < w$ or $w \not \in f(\Lambda)$.
\item
A cycle max if $w_i$ is the (numerically) largest element in a closed word of length at least $2$.
\item
A peak if $f^{-1}(w) < w > f(w)$ and it is not a cycle max.
\end{itemize}
Clearly each letter in $[n]$ belongs to one of these six types. Then a slight extension of the argument in the previous proposition shows that if we define
\[
\begin{split}
& \KS{a_1 \otimes \ldots \otimes a_n \otimes a_{n+1}}
= \KS{a_1 \otimes \ldots \otimes a_n} X(a_{n+1}) - \alpha \KS{a_1 \otimes \ldots \otimes a_n} \Exp{a_{n+1}} \\
&\qquad - \sum_{i=1}^n \beta_1 \KS{a_1 \otimes \ldots \otimes \hat{a}_i \otimes \ldots \otimes a_{n} \otimes a_i a_{n+1}} - \sum_{i=1}^n \beta_2 \KS{a_1, \ldots, \hat{a}_i, \ldots, a_{n+1} a_i} \\
&\qquad - \sum_{i=1}^n t \KS{a_1 \otimes \ldots \otimes \hat{a}_i \otimes \ldots \otimes a_n} \Exp{a_i a_{n+1}} \\
&\qquad - \sum_{1 \leq i < j \leq n} \gamma \KS{a_1 \otimes \ldots \otimes \hat{a}_i \otimes \ldots \otimes \hat{a}_j \otimes \ldots \otimes a_{n} \otimes a_i a_{n+1} a_j} \\
&\qquad - \sum_{1 \leq i < j \leq n} \gamma \KS{a_1 \otimes \ldots \otimes \hat{a}_i \otimes \ldots \otimes \hat{a}_j \otimes \ldots \otimes a_{n} \otimes a_j a_{n+1} a_i},
\end{split}
\]
then
\[
\begin{split}
X(a_1) \ldots X(a_n)
& = \sum_{(\Lambda, f) \in \mc{IPRM}(n)} \KSP{\mc{IPRM}}{a_1 \otimes \ldots \otimes a_n}^{(\Lambda, f)} \\
&\qquad\qquad \times \alpha^{\#\text{closed singletons}} \beta_1^{\#\text{double rises}} \beta_2^{\#\text{double falls}} t^{\#\text{cycle max}} \gamma^{\#\text{peaks}}.
\end{split}
\]
See \cite{Biane-Heine,Simion-Stanton-Laguerre,Clarke-Steingrimsson-Zeng--Euler-Mahonian,KimZeng} for related results. As in most of these references, there is a natural way of including a $q$ parameter in this expansion, based on the values of the $i, j$ indices from the Wick product recursion. However, our technique for obtaining inversion and product formulas does not apply to this extension, and based on the results in Section~\ref{Sec:Counterexample}, it is unclear what these formulas should be.
\end{Remark}

\begin{Remark}
\label{Remark:free-expansions}
Let $\mc{D}$ be a unital star-subalgebra, $\mc{M}$ be a complex star-algebra which is also a $\mc{D}$-bimodule such that for $d_1, d_2 \in \mc{D}$ and $a \in \mc{M}$,
\begin{equation}
\label{Eq:Left-right-action}
d_1 (a d_2) = (d_1 a) d_2,
\end{equation}
and $\Exp{\cdot}: \mc{M} \rightarrow \mc{D}$ a star-linear $\mc{D}$-bimodule map. (We do not assume that $\mc{D} \subset \mc{M}$ since $\mc{M}$ may not be unital.) Let $\Gamma(\mc{M})$ be the complex unital star-algebra generated by non-commuting symbols $\set{X(a): a \in \mc{M}}$ and $\mc{D}$, subject to the linearity relations
\[
X(d_1 a d_2 + d_3 b d_4) = d_1 X(a) d_2 + d_3 X(b) d_4, \quad d_1, d_2, d_3, d_4 \in \mc{D}.
\]
The star-operation on it is determined by the requirement that all $X(a^\ast) = X(a)^\ast$. Thus
\[
\Gamma(\mc{M}) \simeq \bigoplus_{n=0}^\infty \mc{M}^{\otimes_{\mc{D}} n}.
\]
We denote $M(a_1 \otimes \ldots \otimes a_n) = X(a_1) \ldots X(a_n)$, and note that $M$ may be extended to a $\mc{D}$-bimodule map on $\mc{M}^{\otimes_{\mc{D}} n}$.

Let $\pi \in \mc{NC}(n)$. We will define a bimodule map $\Exp{\cdot}^\pi$ on $\mc{M}^{\otimes_{\mc{D}} n}$ recursively as follows. First,
\[
\Exp{d_0 a_1 d_1 \otimes \ldots \otimes a_n d_n}^{\hat{1}_n} = d_0 \Exp{a_1 d_1 \ldots a_n} d_n.
\]
Next let
\[
\Outer(\pi) = \set{V_1 < V_2 < \ldots < V_\ell}, \quad V_i = \set{v(i,1) < \ldots < v(i, t(i))}.
\]
Denote $I_{ij} = [v(i,j) + 1, \ldots, v(i, j+1) - 1]$ for $1 \leq i \leq \ell$, $1 \leq j \leq t(i) - 1$, and $\pi_{i,j} = \pi|_{I_{ij}}$. Note that an interval may be empty. Then we recursively define
\[
\Exp{d_0 a_1 d_1 \otimes \ldots \otimes a_n d_n}^\pi = d_0 \prod_{i=1}^\ell \Exp{ \prod_{j=1}^{t(i) - 1} \left( a_{v(i,j)} d_{v(i,j)} \Exp{a_v d_v: v \in I_{i,j}}^{\pi_{ij}} \right) a_n} d_n.
\]
Note that this is the not the same definition as that in \cite{SpeHab} or Section~3 in \cite{Ans-Bel-Fev-Nica}, although it is related to them and may be expressed in terms of them as long as $\pi$ is appropriately transformed.

Next, let $F$ be a $\mc{D}$-bimodule map on $\mc{M}^{\otimes_{\mc{D}} n}$ (in our examples, either $M$ or $W$). Let $(\pi, S) \in \mc{INC}(n)$, and this time denote
\[
S = \set{V_1 < V_2 < \ldots < V_\ell}, \quad V_i = \set{v(i,1) < \ldots < v(i, t(i))}.
\]
Let $I_{ij}$, $\pi_{ij}$ be as before, and define additionally $v(\ell + 1, 1) = n+1$,
\[
I_{i, t(i)} = [v(i, t(i)) + 1, \ldots, v(i+1, 1) - 1], \quad I_0 = [1, \ldots, v(1,1) - 1],
\]
and the corresponding $\pi_{ij}$, $\pi_0$. Denote
\[
F(d_0 a_1 d_1 \otimes \ldots \otimes a_n d_n)^{(\pi, S)} = d_0 F \left( \Exp{a_v d_v: v \in I_0}^{\pi_0} \bigotimes_{i=1}^\ell \prod_{j=1}^{t(i)} \left( a_{v(i,j)} d_{v(i,j)} \Exp{a_v d_v: v \in I_{i,j}}^{\pi_{ij}} \right) \right).
\]
\end{Remark}

In the scalar-valued case, the following results are known, see Theorem~3.3 in \cite{Effros-Popa} and Proposition~29 in \cite{AnsQLevy} for $q=0$.

\begin{Prop}
\label{Prop:free-expansions-1}
Define $\KSP{\mc{INC}_{1,2}}{a_1 \otimes a_2 \otimes \ldots \otimes a_n}$ recursively by
\[
\begin{split}
\KSP{\mc{INC}_{1,2}}{a_1 \otimes \ldots \otimes a_n \otimes a_{n+1}}
& = \KSP{\mc{INC}_{1,2}}{a_1 \otimes \ldots \otimes a_n} X(a_{n+1}) \\
&\quad - \KSP{\mc{INC}_{1,2}}{a_1 \otimes \ldots \otimes a_{n-1}} \Exp{a_n a_{n+1}}.
\end{split}
\]
Note that $W_{\mc{INC}_{1,2}}$ extends to a $\mc{D}$-bimodule map on each $\mc{M}^{\otimes_{\mc{D}} n}$. Then
\begin{equation}
\label{Eq:X-NC12}
M(a_1 \otimes \ldots \otimes a_n) = \sum_{\pi \in \mc{INC}_{1,2}(n)} \KSP{\mc{INC}_{1,2}}{a_1 \otimes \ldots \otimes a_n}^{(\pi, \Sing(\pi))},
\end{equation}
and
\[
\KSP{\mc{INC}_{1,2}}{a_1 \otimes \ldots \otimes a_n} = \sum_{\pi \in \Int_{1,2}(n)} (-1)^{\abs{\pi} - \abs{\Sing(\pi)}} M(a_1 \otimes \ldots \otimes a_n)^{(\pi, \Sing(\pi))},
\]
and
\begin{multline*}
\prod_{i=1}^k \KSP{\mc{INC}_{1,2}}{a_{u_i(1)} \otimes a_{u_i(2)} \otimes \ldots \otimes a_{u_i(s(i))}} \\
= \sum_{\substack{\pi \in \mc{INC}_{1,2}(n) \\ \pi \wedge (\hat{1}_{s(1)}, \ldots, \hat{1}_{s(k)}) = \hat{0}_n}} \KSP{\mc{INC}_{1,2}}{a_1 \otimes \ldots \otimes a_n}^{(\pi, \Sing(\pi))}.
\end{multline*}
\end{Prop}

The proof is similar to and simpler than that of the next proposition, so we omit it.

\begin{Prop}
\label{Prop:free-expansions-2}
Define $\KSP{\mc{INC}}{a_1 \otimes a_2 \otimes \ldots \otimes a_n}$ recursively by
\[
\begin{split}
& \KSP{\mc{INC}}{a_1 \otimes \ldots \otimes a_n \otimes a_{n+1}} \\
&\quad = \KSP{\mc{INC}}{a_1 \otimes \ldots \otimes a_n} X(a_{n+1}) - \KSP{\mc{INC}}{a_1 \otimes \ldots \otimes a_n} \Exp{a_{n+1}} \\
&\qquad - \KSP{\mc{INC}}{a_1 \otimes \ldots \otimes a_n a_{n+1}} - \KSP{\mc{INC}}{a_1 \otimes \ldots \otimes a_{n-1}} \Exp{a_n a_{n+1}}.
\end{split}
\]
Then
\begin{equation}
\label{Eq:X-INC}
X(a_1) \ldots X(a_n) = \sum_{(\pi, S) \in \mc{INC}(n)} \KSP{\mc{INC}}{a_1 \otimes \ldots \otimes a_n}^{(\pi, S)},
\end{equation}
and
\[
\KSP{\mc{INC}}{a_1 \otimes \ldots \otimes a_n} = \sum_{\substack{(\pi, S) \in \mc{INC}(n) \\ \pi \in \Int(n), U \in \pi \setminus S \Rightarrow \abs{U} = 1}} (-1)^{n - \abs{S}} M(a_1 \otimes \ldots \otimes a_n)^{(\pi, S)},
\]
and
\begin{equation*}
\prod_{i=1}^k \KSP{\mc{INC}}{a_{u_i(1)} \otimes a_{u_i(2)} \otimes \ldots \otimes a_{u_i(s(i))}}
= \sum_{\substack{(\pi, S) \in \mc{INC}(n) \\ \pi \wedge (\hat{1}_{s(1)}, \ldots, \hat{1}_{s(k)}) = \hat{0}_n \\ \Sing(\pi) \subset S}} \KSP{\mc{INC}}{a_1 \otimes \ldots \otimes a_n}^{(\pi, S)}.
\end{equation*}
\end{Prop}

\begin{proof}
The proof of equation~\eqref{Eq:X-INC} is similar to the argument in Theorem~\ref{Thm:5th-expansion} above or Theorem~\ref{Thm:Free-Meixner-Wick} below, so we only outline it. It is based on the observation that the four terms in the recursion relation for $W_{\mc{INC}}$ correspond to the decomposition of $\mc{INC}(n)$ as a disjoint union of four sets: those where $n+1$ is an open singleton, a closed singleton, those where it belongs to larger open block, and those where it belongs to a larger closed block. Equation~\eqref{Eq:X-INC} implies that for $(\pi, S) \in \mc{INC}(n)$,
\[
M(a_1 \otimes \ldots \otimes a_n)^{(\pi, S)} = \sum_{\substack{(\sigma, T) \in \mc{INC}(n) \\ (\sigma, T) \geq (\pi, S)}} \KSP{\mc{INC}}{a_1 \otimes \ldots \otimes a_n}^{(\pi, S)}.
\]
So Theorem~\ref{Thm:General-product} and Proposition~\ref{Prop:INC} imply the results.
\end{proof}

\section{Expansions for free Meixner Wick products}
\label{Sec:Free-Meixner}

\begin{Notation}
A \emph{covered partition} is a partition $\pi \in \mc{NC}(\Lambda)$ with a single outer block, or equivalently such that $\min (\Lambda) \stackrel{\pi}{\sim} \max (\Lambda)$; their set is denoted by $\mc{NC}'(\Lambda)$. We define an additional order on $\mc{NC}(\Lambda)$: $\pi \ll \sigma$ if $\pi \leq \sigma$ and in addition, for each block $U \in \sigma$, $\pi|_U \in \mc{NC}'(U)$. See \cite{Belinschi-Nica-Eta,Nica-NC-Linked} for more details.

For $(\pi, S) \in \mc{INC}(n)$ and $\sigma \ll \pi$, we say that a block of $\sigma$ is open if it contains the smallest element of an open block of $\pi$; their collection is denoted $S'(\sigma, S)$. In particular, each open singleton block of $\sigma$ consists of the smallest element of some open block of $\pi$.
\end{Notation}

\begin{Defn}
\label{Defn:Free-Meixner-Wick}
For $\mc{D}$, $\mc{M}$, $\Gamma(\mc{M})$, and $X(a)$ as in Remark~\ref{Remark:free-expansions}, define the free Meixner-Kailath-Segall polynomials by the recursion
\begin{equation*}
\label{Meixner-KS}
\begin{split}
\KS{a_1 \otimes \ldots \otimes a_n \otimes a_{n+1}}
& = \KS{a_1 \otimes \ldots \otimes a_n} X(a_{n+1}) - \alpha \KS{a_1 \otimes \ldots \otimes a_n} \Exp{a_{n+1}} \\
& - \beta \KS{a_1 \otimes \ldots \otimes a_{n-1} \otimes a_n a_{n+1}} - t \KS{a_1 \otimes \ldots \otimes a_{n-1}} \Exp{a_n a_{n+1}} \\
& - \gamma \KS{a_1 \otimes \ldots \otimes a_{n-2} \otimes a_{n-1} a_n a_{n+1}}
\end{split}
\end{equation*}
and in particular
\begin{equation*}
\KS{a_1 \otimes a_2} = \KS{a_1} X(a_2) - \alpha \KS{a_1} \Exp{a_2} - \beta \KS{a_1 a_2} - t \Exp{a_1 a_2}
\end{equation*}
and $\KS{a_1} = X(a_1) - \alpha \Exp{a_1}$. Compare with Section~7 in \cite{Sniady-SWN}.
\end{Defn}

\begin{Notation}
For $(\pi, S) \in \mc{INC}(n)$, let
\[
C^{(\pi, S)}_{\alpha, \beta, t, \gamma}
= \sum_{\substack{\sigma \leq \pi, \\ U \in \pi \setminus S \Rightarrow \sigma|_U \in \mc{NC}'(U), \\ \Sing(\sigma) \subset \Sing(\pi) \cup S'(\sigma, S)}}  \alpha^{\abs{\Sing(\pi \setminus S)}} \beta^{n - 2 \abs{\sigma} + \abs{S} + \abs{\Sing(\pi \setminus S)}} t^{\abs{\pi \setminus S} - \abs{\Sing(\pi \setminus S)}} \gamma^{\abs{\sigma} - \abs{\pi}}.
\]

In particular,
\[
\begin{split}
C^{\pi}_{\alpha, \beta, t, \gamma}
= C^{(\pi, \emptyset)}_{\alpha, \beta, t, \gamma}
= \sum_{\substack{\sigma \ll \pi, \\ \Sing(\sigma) = \Sing(\pi)}}  \alpha^{\abs{\Sing(\pi)}} \beta^{n - 2 \abs{\sigma} + \abs{\Sing(\pi)}} t^{\abs{\pi} - \abs{\Sing(\pi)}} \gamma^{\abs{\sigma} - \abs{\pi}}.
\end{split}
\]
\end{Notation}

\begin{Lemma}
\label{Lemma:Kappa}
Denote $M_n(\beta, \gamma)$ a particular case of the Jacobi-Rogers polynomials, the sum over Motzkin paths of length $n$ with flat steps given weight $\beta$ and down steps given weight $\gamma$. Then
\[
C^{(\pi, S)}_{\alpha, \beta, t, \gamma}
= \prod_{U \in \pi \setminus S} \kappa^{\abs{U}}_{\alpha, \beta, t, \gamma} \prod_{V \in S}  \omega^{\abs{V}}_{\alpha, \beta, t, \gamma},
\]
where
\[
\omega^n_{\alpha, \beta, t, \gamma} = M_{n-1}(\beta, \gamma), \quad \kappa^1_{\alpha, \beta, t, \gamma} = \alpha, \quad \kappa^n_{\alpha, \beta, t, \gamma}
= t M_{n-2}(\beta, \gamma).
\]
\end{Lemma}

\begin{proof}
We note that
\[
\kappa^1_{\alpha, \beta, t, \gamma} = \alpha,
\quad \kappa^n_{\alpha, \beta, t, \gamma}
= t \sum_{\tau \in \mc{NC}_{\geq 2}'(n)} \beta^{n - 2 \abs{\tau}} \gamma^{\abs{\tau} - 1}
= t M_{n-2}(\beta, \gamma),
\]
and
\[
\omega^1_{\alpha, \beta, t, \gamma} = 1,
\quad \omega^n_{\alpha, \beta, t, \gamma}
= \sum_{\substack{\tau \in \mc{NC}(n), \\ \Sing(\tau) \subset \set{1}}} \beta^{n - 2 \abs{\tau} + 1} \gamma^{\abs{\tau} - 1}
= \sum_{\tau \in \mc{NC}_{\geq 2}'(n+1)} \beta^{n - 2 \abs{\tau} + 1} \gamma^{\abs{\tau} - 1}
= M_{n-1}(\beta, \gamma). \qedhere
\]
\end{proof}

We formulate and prove the following theorem for the case $\mc{D} = \mf{C}$ to simplify notation, but the result carries over verbatim for general $\mc{D}$.

\begin{Thm}
\label{Thm:Free-Meixner-Wick}
We have expansions of monomials
\begin{equation}
\label{Eq:Monomial}
X(a_1) X(a_2) \ldots X(a_n) = \sum_{(\pi, S) \in \mc{INC}(n)} C^{(\pi, S)}_{\alpha, \beta, t, \gamma} \prod_{U \in \pi \setminus S} \Exp{a_U} \KS{\bigotimes_{V \in S} a_V}.
\end{equation}
\end{Thm}

\begin{proof}
By induction
\begin{equation*}
\begin{split}
& \prod_{i=1}^n X(a_i) X(a_{n+1}) \\
&\quad = \sum_{(\pi, S) \in \mc{INC}(n)} C^{(\pi, S)}_{\alpha, \beta, t, \gamma} \prod_{U \in \pi \setminus S} \Exp{a_U} \KS{\bigotimes_{V \in S} a_V} X(a_{n+1}) \\
&\quad = \sum_{(\pi, S) \in \mc{INC}(n)} \sum_{\substack{\sigma \leq \pi, \\ U \in \pi \setminus S \Rightarrow \sigma|_U \in \mc{NC}'(U), \\ \Sing(\sigma) \subset \Sing(\pi) \cup S'(\sigma, S)}} \alpha^{\abs{\Sing(\pi \setminus S)}} \beta^{n - 2 \abs{\sigma} + \abs{S} + \abs{\Sing(\pi \setminus S)}} t^{\abs{\pi \setminus S} - \abs{\Sing(\pi \setminus S)}} \gamma^{\abs{\sigma} - \abs{\pi}} \prod_{U \in \pi \setminus S} \Exp{a_U} \\
&\qquad\qquad \Biggl(\KS{a_{V_1} \otimes a_{V_2} \otimes \ldots \otimes a_{V_j} \otimes a_{n+1} : \set{V_1 < V_2 < \ldots < V_j} = S} \\
&\qquad\qquad + \alpha \KS{a_{V_1} \otimes a_{V_2} \otimes \ldots \otimes a_{V_j} : \set{V_1 < V_2 < \ldots < V_j} = S} \Exp{a_{n+1}} \\
&\qquad\qquad + \beta \chf{\abs{S} \geq 1} \KS{a_{V_1} \otimes a_{V_2} \otimes \ldots \otimes a_{V_j} a_{n+1} : \set{V_1 < V_2 < \ldots < V_j} = S} \\
&\qquad\qquad + t \chf{\abs{S} \geq 1} \KS{a_{V_1} \otimes a_{V_2} \otimes \ldots \otimes a_{V_{j-1}} : \set{V_1 < V_2 < \ldots < V_j} = S} \Exp{a_{V_j} a_{n+1}} \\
&\qquad\qquad + \gamma \chf{\abs{S} \geq 2} \KS{a_{V_1} \otimes a_{V_2} \otimes \ldots \otimes a_{V_{j-1}} a_{V_j} a_{n+1} : \set{V_1 < V_2 < \ldots < V_j} = S} \Biggr).
\end{split}
\end{equation*}
For fixed $(\pi, S, \sigma)$, the first term produces all the triples $(\pi', S', \sigma')$ with
\[
(\pi', S') \in \mc{INC}(n+1),
\sigma' \leq \pi', U \in \pi' \setminus S' \Rightarrow \sigma'|_U \in \mc{NC}'(U), \Sing(\sigma') \subset \Sing(\pi') \cup S'(\sigma', S')
\]
in which $n+1$ is an open singleton in $\pi'$ (and so in $\sigma'$). Since $n$, $\abs{\pi}$, $\abs{S}$, and $\abs{\sigma}$ are all incremented by $1$, $C^{(\pi, S)}_{\alpha, \beta, t, \gamma}$ does not change. The second term produces all triples in which $n+1$ is a closed singleton in $\pi'$ (and so in $\sigma'$). Since $n$, $\abs{\pi}$, $\abs{\pi \setminus S}$, $\abs{\Sing(\pi \setminus S)}$, and $\abs{\sigma}$ are all incremented by $1$, $C^{(\pi, S)}$ is multiplied by $\alpha$. The third term produces all triples in which $n+1$ belongs to an open block in both $\sigma'$ and $\pi'$, each of size at least $2$, by adjoining it to the largest open block of $\pi$ and the corresponding open block of $\sigma$. Since only $n$ is incremented, $C^{(\pi, S)}$ is multiplied by $\beta$. The fourth term produces all triples in which $n+1$ belongs to a closed non-singleton block of $\pi'$ (and so also of $\sigma'$), by adjoining it to the largest open block of $\pi$ and the corresponding open block of $\sigma$, and closing them both. Since $n$ is incremented by $1$ and $\abs{S}$ decreased by $1$, $C^{(\pi, S)}$ is multiplied by $t$. The fifth term produces all triples in which $n+1$ belongs to an open block of $\pi'$ but a closed block of $\sigma'$, by adjoining $n+1$ to the second largest open block of $\pi$ and the corresponding open block of $\sigma$, combining the two largest open blocks of $\pi$, and closing the open block of $\sigma$ which belonged to the largest open block of $\pi$. Since $n$ is incremented by $1$ while $\abs{\pi}$ and $\abs{S}$ are decreased by $1$, $C^{(\pi, S)}$ is multiplied by $\gamma$. These five classes are disjoint and exhaust the triples $(\pi', S', \sigma')$ above.
\end{proof}

\begin{Cor}
For the state corresponding to the Wick products from Definition~\ref{Defn:Free-Meixner-Wick}, the joint moments are
\begin{equation}
\label{Eq:Moment}
\state{X(a_1) X(a_2) \ldots X(a_n)} = \sum_{\pi \in \mc{NC}(n)} C^{\pi}_{\alpha, \beta, t, \gamma} \Exp{a_1 \otimes \ldots \otimes a_n}^\pi.
\end{equation}
\end{Cor}

\begin{Remark}
\label{Remark:Free-cumulants}
Using Lemma~\ref{Lemma:Kappa}, we may re-write formula~\eqref{Eq:Moment} as
\[
\state{X(a_1) X(a_2) \ldots X(a_n)}
= \sum_{\pi \in \mc{NC}(n)} \left( \prod_{U \in \pi} \kappa_{\alpha, \beta, t, \gamma}^{\abs{U}} \right) \Exp{a_1 \otimes \ldots \otimes a_n}^\pi.
\]
Therefore by definition, the joint free cumulants of $\set{X(a_i)}$ are
\[
R[X(a_1)] = \kappa_{\alpha, \beta, t, \gamma}^{1} \Exp{a_1} = \alpha \Exp{a_1},
\]
\[
R[X(a_1), \ldots, X(a_n)] = \kappa_{\alpha, \beta, t, \gamma}^{n} \Exp{a_1 \ldots a_n} = t M_{n-2}(\beta, \gamma) \Exp{a_1 \ldots a_n}.
\]
Compare with Theorem~8 in \cite{Sniady-SWN}. Note that $M_{n-2}(1, 1) = M_{n-2}$, the Motzkin number. On the other hand,
\[
M_{n-2}(2, 1) = \sum_{\tau \in \mc{NC}_{\geq 2}'(n)} 2^{n - 2 \abs{\tau}} = \abs{\mc{NC}'(n)} = c_{n-1},
\]
the Catalan number. In general $M_n(\beta, \gamma)$ are the moments of a semicircular distribution with mean $\beta$ and variance $\gamma$. Cf. Theorem~2 in \cite{AnsFree-Meixner}.
\end{Remark}

The following proposition is, roughly speaking, taken as the definition in \cite{Sniady-SWN}.

\begin{Prop}
\label{Prop:Inner-product}
\begin{multline*}
\state{\KS{a_1 \otimes \ldots \otimes a_{n}} \KS{b_k \otimes \ldots \otimes b_1}} \\
= \delta_{n=k} \sum_{\pi \in \Int(n)} t^{\abs{\pi}} \gamma^{n - \abs{\pi}} \Exp{a_1 \otimes \ldots \otimes a_n \otimes b_k \otimes \ldots \otimes b_1}^{\set{U \cup (2n+1-U) : U \in \pi}}.
\end{multline*}
\end{Prop}

\begin{proof}
For $n = 0$ and arbitrary $k$, the result follows from the definition of $\phi$. So it suffices to show that the result for $(u, v) \leq (n-1, k+1)$ implies the result for $(n, k)$. For $n \geq 1$,
\begin{equation}
\label{Eq:Inner-product-1}
\begin{split}
& \state{\KS{a_1 \otimes \ldots \otimes a_{n}} \KS{b_k \otimes \ldots \otimes b_1}} \\
&\quad = \state{\KS{a_1 \otimes \ldots \otimes a_{n-1}} X(a_{n}) \KS{b_k \otimes \ldots \otimes b_1}} \\
&\quad\quad - \alpha \state{\KS{a_1 \otimes \ldots \otimes a_{n-1}} \Exp{a_{n}} \KS{b_k \otimes \ldots \otimes b_1}} \\
&\quad\quad - \beta \state{\KS{a_1 \otimes \ldots \otimes a_{n-2} \otimes a_{n-1} a_{n}} \KS{b_k \otimes \ldots \otimes b_1}} \\
&\quad\quad - t \state{\KS{a_1 \otimes \ldots \otimes a_{n-2}} \Exp{a_{n-1} a_{n}} \KS{b_k \otimes \ldots \otimes b_1}} \\
&\quad\quad - \gamma \state{\KS{a_1 \otimes \ldots \otimes a_{n-3} \otimes a_{n-2} a_{n-1} a_{n}} \KS{b_k \otimes \ldots \otimes b_1}} \\
\end{split}
\end{equation}
Applying the recursion in Definition~\ref{Defn:Free-Meixner-Wick} to the first term (and using the adjoint symmetry in Proposition~\ref{Prop:Tracial} below), this term equals
\begin{equation*}
\label{Eq:Inner-product-2}
\begin{split}
&\qquad  \state{\KS{a_1 \otimes \ldots \otimes a_{n-1}} \KS{a_{n} \otimes  b_k \otimes \ldots \otimes b_1}} \\
&+ \alpha \state{\KS{a_1 \otimes \ldots \otimes a_{n-1}} \Exp{a_{n}} \KS{b_k \otimes \ldots \otimes b_1}} \\
&+ \beta \state{\KS{a_1 \otimes \ldots \otimes a_{n-1}} \KS{a_{n} b_k \otimes \ldots \otimes b_1}} \\
&+ t \state{\KS{a_1 \otimes \ldots \otimes a_{n-1}} \Exp{a_{n} b_k} \KS{b_{k-1} \otimes \ldots \otimes b_1}} \\
&+ \gamma \state{\KS{a_1 \otimes \ldots \otimes a_{n-1}} \KS{a_{n} b_k b_{k-1} \otimes b_{k-2} \ldots \otimes b_1}}.
\end{split}
\end{equation*}
This the expression~\eqref{Eq:Inner-product-1} equals
\begin{equation*}
\begin{split}
&\qquad \delta_{n-1=k+1} \sum_{\pi \in \Int(n-1)} t^{\abs{\pi}} \gamma^{n - \abs{\pi}} \Exp{a_1 \otimes \ldots \otimes a_n \otimes b_k \otimes \ldots \otimes b_1}^{\set{U \cup (2n-1-U) : U \in \pi}} \\
&+ \beta \delta_{n-1=k} \sum_{\pi \in \Int(n-1)} t^{\abs{\pi}} \gamma^{n - \abs{\pi}} \Exp{a_1 \otimes \ldots \otimes a_{n-1} \otimes a_n b_k \otimes \ldots \otimes b_1}^{\set{U \cup (2n-1-U) : U \in \pi}} \\
&+ t \Exp{a_n b_k} \delta_{n=k} \sum_{\pi \in \Int(n-1)} t^{\abs{\pi}} \gamma^{n - \abs{\pi}} \Exp{a_1 \otimes \ldots \otimes a_{n-1} \otimes b_{k-1} \otimes \ldots \otimes b_1}^{\set{U \cup (2n-1-U) : U \in \pi}} \\
&+ \gamma \delta_{n=k} \sum_{\pi \in \Int(n-1)} t^{\abs{\pi}} \gamma^{n - \abs{\pi}} \Exp{a_1 \otimes \ldots \otimes a_{n-1} \otimes a_n b_k b_{k-1} \otimes b_{k-2} \otimes \ldots \otimes b_1}^{\set{U \cup (2n-1-U) : U \in \pi}} \\
&- \beta \delta_{n-1=k} \sum_{\pi \in \Int(n-1)} t^{\abs{\pi}} \gamma^{n - \abs{\pi}} \Exp{a_1 \otimes \ldots \otimes a_{n-1} a_n \otimes b_k \otimes \ldots \otimes b_1}^{\set{U \cup (2n-1-U) : U \in \pi}} \\
&- t \Exp{a_{n-1} a_n} \delta_{n-2=k} \sum_{\pi \in \Int(n-2)} t^{\abs{\pi}} \gamma^{n - \abs{\pi}} \Exp{a_1 \otimes \ldots \otimes a_{n-2} \otimes b_k \otimes \ldots \otimes b_1}^{\set{U \cup (2n-3-U) : U \in \pi}} \\
&- \gamma \delta_{n-2=k} \sum_{\pi \in \Int(n-2)} t^{\abs{\pi}} \gamma^{n - \abs{\pi}} \Exp{a_1 \otimes \ldots \otimes a_{n-2} a_{n-1} a_n \otimes b_k \otimes \ldots \otimes b_1}^{\set{U \cup (2n-3-U) : U \in \pi}} \\
\end{split}
\end{equation*}
The second and fifth sums cancel term-by-term. Next, suppose $n = k+2$. If in a partition $\pi \in \Int(n-1)$ in the first sum, $(n-1)$ is a singleton, the term corresponding to this partition cancels with the corresponding term in the sixth sum. If $(n-1)$  is not a singleton, the term corresponding to this partition cancels with the corresponding term in the seventh sum. The sum of the remaining sums (third and fourth), for $n = k$, equals
\[
\delta_{n=k} \sum_{\pi \in \Int(n)} t^{\abs{\pi}} \gamma^{n - \abs{\pi}} \Exp{a_n^\ast \otimes \ldots \otimes a_1^\ast \otimes b_1 \otimes \ldots \otimes b_n}^{\set{U \cup (2n+1-U) : U \in \pi}}
\]
by the same decomposition.
\end{proof}

\begin{Thm}
\label{Thm:Inversion}
We may expand
\begin{equation}
\label{Eq:free-Meixner-expansion}
\KS{a_1 \otimes \ldots \otimes a_n} = \sum_{\substack{(\pi, S) \\ \pi \in \Int(n) \\ S \subset \pi}} (-1)^{n - \abs{S}} \prod_{U \in \pi \setminus S} c_{\abs{U}} \Exp{a_U} \prod_{V \in S} o_{\abs{V}} X(a_V),
\end{equation}
where
\[
c_k = \alpha o_k - t o_{k-1}.
\]
Case I: $\gamma = 0$. Then
\[
o_k = \beta^{k-1}, \qquad c_k = (\alpha \beta - t) \beta^{k-2}
\]
For $\gamma \neq 0$, factor $1 - \beta z + \gamma z^2 = (1 - u z)(1 - v z)$.

Case II: $\gamma \neq 0$, $\beta^2 \neq 4 \gamma$, so that $u \neq v$. Then
\[
o_k = \frac{1}{u - v} (u^k - v^k), \qquad c_k = \frac{1}{u-v} \left( \alpha (u^k - v^k) - t (u^{k-1} - v^{k-1}) \right).
\]
Case II$'$: if in addition, $\alpha^2 - \alpha \beta t + \gamma t^2 = 0$, so that $v = t/\alpha$, then
\[
o_k = \frac{1}{\beta - 2t/\alpha} ((\beta - t/\alpha)^k - (t/\alpha)^k), \qquad c_k = \alpha (\beta - t/\alpha)^{k-1}.
\]
Case III: $\gamma \neq 0$, $\beta^2 = 4 \gamma$, so that $u = v = \beta/2$. Then
\[
o_k = k (\beta/2)^{k-1}, \qquad c_k = (\alpha k (\beta/2) - t (k-1)) (\beta/2)^{k-2}.
\]
Case III$'$: if in addition, $\alpha \beta = 2 t$, so that $u = v = t/\alpha$, then
\[
o_k = k (\beta/2)^{k-1}, \qquad c_k = \alpha (\beta/2)^{k-1}.
\]
\end{Thm}

\begin{proof}
Write $\KS{a_1 \otimes \ldots \otimes a_n}$ in the form \eqref{Eq:free-Meixner-expansion}; we will show that this is possible by exhibiting coefficients in this expansion. Plugging in this expansion into the recursion in Definition~\ref{Defn:Free-Meixner-Wick}, we obtain
\[
\begin{split}
\sum_{\substack{(\pi, S) \\ \pi \in \Int(n+1) \\ S \subset \pi}} (-1)^{n + 1 - \abs{S}}
& = \sum_{\substack{(\pi, S) \\ \pi \in \Int(n) \\ S \subset \pi}} (-1)^{n - \abs{S}} X(a_{n+1})
- \alpha \sum_{\substack{(\pi, S) \\ \pi \in \Int(n) \\ S \subset \pi}} (-1)^{n - \abs{S}} \Exp{a_{n+1}} \\
& - \beta \sum_{\substack{(\pi, S) \\ \pi \in \Int([n-1] \cup \set{\set{n, n+1}}) \\ S \subset \pi}} (-1)^{n - \abs{S}} \\
& - t \sum_{\substack{(\pi, S) \\ \pi \in \Int(n-1) \\ S \subset \pi}} (-1)^{n - 1 - \abs{S}} \Exp{a_n a_{n+1}}
- \gamma \sum_{\substack{(\pi, S) \\ \pi \in \Int([n-2] \cup \set{\set{n-1, n, n+1}}) \\ S \subset \pi}} (-1)^{n - 1 - \abs{S}},
\end{split}
\]
where in each term we sum the expression
\[
\prod_{U \in \pi \setminus S} c_{\abs{U}} \Exp{a_U} \prod_{V \in S} o_{\abs{V}} X(a_V).
\]
Now compare the factors corresponding to the block $B$ containing $n+1$ on the left-hand-side. If $B$ is an open singleton, it matches with a term in the first sum, with the same coefficient (since the number of open blocks on the left is one more than on the right). Thus $o_1 = 1$. For the remaining terms, the size of $S$ does not change, so we omit it from the coefficients. If $B$ is a closed singleton, it matches with a term from the second sum, and the coefficients are $(-1)^{n+1} c_1 = (-1)^n (-\alpha)$, so $c_1 = \alpha$. If $B$ is an open pair, it matches with a term in the third sum, and the coefficients are $(-1)^{n+1} o_2 = (-1)^n (- \beta) o_1$, so $o_2 = \beta o_1$. If $B$ is a closed pair, it matches with terms in the third and the fourth sums, and the coefficients are $(-1)^{n+1} c_2 = (-1)^n (- \beta) c_1 + (-1)^{n-1} (-t) $, so $c_2 = \beta c_1 - t$. If $B$ is a larger block, it matches with terms in the third and the fifth sums, and the coefficients are $(-1)^{n+1} o_k = (-1)^n (- \beta) o_{k-1} + (-1)^{n-1} (- \gamma) o_{k-2}$ (and the corresponding expression for $c_k$), so that
\[
o_k = \beta o_{k-1} - \gamma o_{k-2}, \qquad c_k = \beta c_{k-1} - \gamma c_{k-2}
\]
for $k \geq 3$. Let
\[
O(z) = \sum_{k=1}^\infty o_k z^{k-1}, \qquad C(z) = \sum_{k=1}^\infty c_k z^{k-1}.
\]
Then
\[
O(z) = 1 + \beta z O(z) - \gamma z^2 O(z), \qquad C(z) = \alpha - t z + \beta z C(z) - \gamma z^2 C(z),
\]
so
\[
O(z) = \frac{1}{1 - \beta z + \gamma z^2}, \qquad C(z) = \frac{\alpha - t z}{1 - \beta z + \gamma z^2} = \alpha O(z) - t z O(z).
\]
The specific cases follow.
\end{proof}

\section{Counterexample}
\label{Sec:Counterexample}

The following is Definition~4.9 from \cite{AnsAppell}. Here $\mc{M}, \Gamma(\mc{M})$ are as in the introduction.

\begin{Defn}
For $a_i \in \mc{M}^{sa}$, define the $q$-Kailath-Segall polynomials by $\KSP{q}{a} = X(a) - \Exp{a}$ and
\begin{equation}
\label{q-KS}
\begin{split}
\KSP{q}{a, a_1, a_2, \ldots, a_n}
& = X(a) \KSP{q}{a_1, a_2, \ldots, a_n}
- \sum_{i=1}^n q^{i-1} \Exp{a a_i} \KSP{q}{a_1, \ldots, \hat{a}_i, \ldots, a_n} \\
&\quad - \sum_{i=1}^n q^{i-1} \KSP{q}{a a_i, \ldots, \hat{a}_i, \ldots, a_n}
- \Exp{a} \KSP{q}{a_1, a_2, \ldots, a_n}.
\end{split}
\end{equation}
This map has a $\mf{C}$-linear extension, so that each $W$ is really a multi-linear map from $\mc{M}$ to $\Gamma(\mc{M})$.
\end{Defn}

\begin{Example}
According to Corollary~4.13 from \cite{AnsAppell},
\[
\statep{q}{\KSP{q}{a_0} \KSP{q}{a_1, a_2, a_3} \KSP{q}{a_4}} = 0.
\]
However a direct calculation shows that in fact
\[
\statep{q}{\KSP{q}{a_0} \KSP{q}{a_1, a_2, a_3} \KSP{q}{a_4}}
= (q - q^2) (\Exp{a_0 a_2} \Exp{a_1 a_3 a_4} - \Exp{a_0 a_2 a_4} \Exp{a_1 a_3}).
\]
To be completely explicit, we consider the case where $a_0 = a_2 = \chf{I}$, $a_1 = a_3 = a_4 = \chf{J}$, $I \cap J = \emptyset$, and the state is the Lebesgue measure. Then we get
\[
\statep{q}{\KSP{q}{a_0} \KSP{q}{a_1, a_2, a_3} \KSP{q}{a_4}}
= (q - q^2) \abs{I} \cdot \abs{J}.
\]
Thus Corollary~4.13, and so also Theorem~4.11 part (c) in \cite{AnsAppell}, are false.
\end{Example}

The formula in Theorem~4.11(c) is true if the arguments of each $W$ are orthogonal; however this does not imply the general result since $\phi_q$ is not tracial. See Remark~\ref{Remark:Ito-isometry}. There are many particular cases when 4.11(c) is true. For the case $q=1$ (classical), and $q=0$ (free), the proof provided in \cite{AnsAppell} still works. For the $q$-Gaussian case, this is Theorem~3.3 in \cite{Effros-Popa}. Finally, for univariate polynomials obtained for equal idempotent $a$ and general $q$, the linearization formulas in Corollary~4.13 also hold \cite{Kim-Stanton-Zeng,Ismail-Kasraoui-Separation}.

\section{Representations and completions}
\label{Sec:Completions}

Let $\mc{M}$ and $\mc{B}$ be $\mc{D}$-bimodules with the actions satisfying \eqref{Eq:Left-right-action}. For a linear $\mc{D}$-bimodule map $F : \mc{M} \rightarrow \mc{B}$, define the map $\mc{F}(F) : \mc{M}^{\otimes_{\mc{D}} n} \rightarrow \mc{B}^{\otimes_{\mc{D}} n}$ by
\[
\mc{F}(F)[d_0 a_1 d_1 \otimes \ldots \otimes a_n d_n] = F(d_0 a_1 d_1) \otimes \ldots \otimes F(a_n d_n),
\]
and the map $\Gamma(F) : \Gamma(\mc{M}) \rightarrow \Gamma(\mc{B})$ by $\Gamma(F)[d] = d$ for $d \in \mc{D}$ and
\[
\Gamma(F) [\KSP{\mc{INC}}{\mb{a}}] = \KSP{\mc{INC}}{\mc{F}(F)[\mb{a}]}.
\]

\begin{Defn}
Let $\mc{M}$ be a star-algebra and $\mc{B}$ a star-subalgebra. An \emph{algebraic conditional expectation} is a star-linear $\mc{B}$-bimodule map $F: \mc{M} \rightarrow \mc{B}$ such that $F^2 = F$. If $\mc{M}$ is a $\mc{D}$-bimodule, $\mc{D} \mc{B} \mc{D} \subset \mc{B}$, and $\tau : \mc{M} \rightarrow \mc{D}$ is a star-linear functional, we say that $F$ preserves $\tau$ if $\tau[F(a)] = \tau[a]$ for $a \in \mc{M}$.
\end{Defn}

\begin{Prop}
Let $\mc{D}, \mc{B}, \mc{M}$ be as in the preceding definition, and $F : \mc{M} \rightarrow \mc{B}$ an algebraic conditional expectation preserving $\Exp{\cdot}$. Then $\Gamma(F) : \Gamma(\mc{M}) \rightarrow \Gamma(\mc{B})$ is an algebraic conditional expectation preserving $\phi_{\mc{INC}}$.
\end{Prop}

\begin{proof}
Clearly $\Gamma(F)$ is the identity on $\Gamma(\mc{B})$. For $\mb{a} \in \mc{B}^{\otimes_{\mc{D}} n}$, $\mb{b} \in \mc{M}^{\otimes_{\mc{D}} k}$, $\mb{c} \in \mc{B}^{\otimes_{\mc{D}} \ell}$,
\[
\begin{split}
& \Gamma(F) \left[ \KSP{\mc{INC}}{\mb{a}} \KSP{\mc{INC}}{\mb{b}} \KSP{\mc{INC}}{\mb{c}} \right] \\
&\quad = \sum_{\substack{(\pi, S) \in \mc{INC}(n+k+\ell) \\ \pi \wedge (\hat{1}_n, \hat{1}_k, \hat{1}_\ell) = \hat{0}_{n+k+\ell} \\ \Sing(\pi) \subset S}} \Gamma(F) \left[ \KSP{\mc{INC}}{\mb{a} \otimes \mb{b} \otimes \mb{c}}^{(\pi, S)} \right] \\
&\quad = \sum_{\substack{(\pi, S) \in \mc{INC}(n+k+\ell) \\ \pi \wedge (\hat{1}_n, \hat{1}_k, \hat{1}_\ell) = \hat{0}_{n+k+\ell} \\ \Sing(\pi) \subset S}} \KSP{\mc{INC}}{\mb{a} \otimes \mc{F}(F)[\mb{b}] \otimes \mb{c}}^{(\pi, S)} \\
&\quad = \KSP{\mc{INC}}{\mb{a}} \Gamma(F) \left[ \KSP{\mc{INC}}{\mb{b}} \right] \KSP{\mc{INC}}{\mb{c}},
\end{split}
\]
where we have used the inhomogeneity of the partitions, the bimodule property of $F$ for open blocks, and both properties of $F$ for closed blocks. The final property is clear.
\end{proof}

\begin{Prop}
\label{Prop:Tracial}
In all six examples above,
\[
\KS{a_1 \otimes a_2 \otimes \ldots \otimes a_n}^\ast = \KS{a_n^\ast \otimes \ldots \otimes a_2^\ast \otimes a_1^\ast},
\]
where for $\phi_{\mc{IP}}$ we additionally assume that $\mc{M}$ is commutative. If $\mc{D} = \mf{C}$ and the linear functional $\Exp{\cdot}$ on $\mc{M}$ is tracial, all six linear functionals $\phi$ are tracial. If $\mc{D}$ is a unital $C^\ast$-algebra, and $\Exp{\cdot}$ is positive, the functionals $\phi$ are positive, where for $\phi_{\mc{IP}}$ and $\phi_{\mc{IPRM}}$ we additionally assume that $\mc{M}$ is commutative.
\end{Prop}

\begin{proof}
The trace and adjoint properties follow from the moment formulas and expansions of Wick products in terms of monomials, since in all cases the coefficients in the expansions depend only on the size of the blocks. For positivity,
\[
\begin{split}
& \statep{\mc{INC}_{1,2}}{\KSP{\mc{INC}_{1,2}}{a_1 \otimes \ldots \otimes a_n}^\ast \KSP{\mc{IP}_{1,2}}{b_1 \otimes \ldots \otimes b_k}} \\
&\quad = \statep{\mc{INC}}{\KSP{\mc{INC}}{a_1 \otimes \ldots \otimes a_n}^\ast \KSP{\mc{INC}}{b_1 \otimes \ldots \otimes b_k}} \\
&\quad = \delta_{n=k} \Exp{a_n^\ast \otimes \ldots \otimes a_1^\ast \otimes b_1 \otimes \ldots \otimes b_n}^{\set{(1, 2n), (2, 2n-1), \ldots, (n, n+1)}},
\end{split}
\]
The proof of positivity of this inner product on $\mc{M}^{\otimes_{\mc{D}} n}$ (which we denote $\ip{\cdot}{\cdot}_n$) is almost verbatim the argument in Theorem 3.5.6 of \cite{SpeHab}. Also, by Proposition~\ref{Prop:Inner-product},
\begin{equation}
\label{Eq:Inner-product-Meixner}
\begin{split}
& \state{\KS{a_1 \otimes \ldots \otimes a_n}^\ast \KS{b_1 \otimes \ldots \otimes b_k}} \\
&\quad = \delta_{n=k} \sum_{\pi \in \Int(n)} t^{\abs{\pi}} \gamma^{n - \abs{\pi}} \Exp{a_n^\ast \otimes \ldots \otimes a_1^\ast \otimes b_1 \otimes \ldots \otimes b_n}^{\set{U \cup (2n+1-U) : U \in \pi}} \\
&\quad = \delta_{n=k} \sum_{\pi \in \Int(n)} t^{\abs{\pi}} \gamma^{n - \abs{\pi}} \ip{\bigotimes_{U \in \pi} a_U}{\bigotimes_{U \in \pi} b_U}_{\abs{\pi}},
\end{split}
\end{equation}
and so this inner product is also positive. For commutative $\mc{M}$,
\[
\begin{split}
& \statep{\mc{IP}_{1,2}}{\KSP{\mc{IP}_{1,2}}{a_1 \otimes \ldots \otimes a_n}^\ast \KSP{\mc{IP}_{1,2}}{b_1 \otimes \ldots \otimes b_k}} \\
&\quad = \statep{\mc{IP}}{\KSP{\mc{IP}}{a_1 \otimes \ldots \otimes a_n}^\ast \KSP{\mc{IP}}{b_1 \otimes \ldots \otimes b_k}}
= \delta_{n=k} \sum_{\alpha \in \Sym(n)} \Exp{a_{\alpha(1)}^\ast b_1} \ldots \Exp{a_{\alpha(n)}^\ast b_n},
\end{split}
\]
This inner product on $\mc{M}^{\otimes n}$ is well known to be positive semi-definite. Finally,
\[
\begin{split}
& \statep{\mc{IPRM}}{\KSP{\mc{IPRM}}{a_1 \otimes \ldots \otimes a_n}^\ast \KSP{\mc{IPRM}}{b_1 \otimes \ldots \otimes b_k}} \\
&\quad = \delta_{n=k} \sum_{\alpha, \beta \in \Sym(n)} \prod_{U \in \pi(\beta)} \Exp{\prod_{i \in U} \left( a_{\alpha(i)}^\ast b_i \right)},
\end{split}
\]
where $\pi(\beta)$ is the orbit decomposition of $\beta$ and the order in each $U \in \pi(\beta)$ is according to $\beta$ as in Theorem~\ref{Thm:5th-expansion}. As observed in Section~4 of \cite{Sniady-SWN}, this inner product is in general not positive. If $\mc{M}$ is commutative, we may re-write
\[
\begin{split}
& \statep{\mc{IPRM}}{\KSP{\mc{IPRM}}{a_1 \otimes \ldots \otimes a_n}^\ast \KSP{\mc{IPRM}}{a_1 \otimes \ldots \otimes a_k}} \\
&\quad = \delta_{n=k} \sum_{\alpha \in \Sym(n)} \sum_{\pi \in \mc{P}(n)} \prod_{U \in \pi} (\abs{U} - 1)! \Exp{\left(\prod_{i \in U} a_{\alpha(i)} \right)^\ast \left( \prod_{i \in U} a_i \right)} \\
&\quad = \delta_{n=k} \frac{1}{n!}  \sum_{\pi \in \mc{P}(n)} \sum_{\alpha, \beta \in \Sym(n)} \prod_{U \in \pi} (\abs{U} - 1)! \Exp{\left(\prod_{i \in U} a_{\alpha(i)} \right)^\ast \left( \prod_{i \in U} a_{\beta(i)} \right)} \\
&\quad = \delta_{n=k} \frac{1}{n!}  \sum_{\pi \in \mc{P}(n)} \prod_{U \in \pi} (\abs{U} - 1)! \Exp{\left(\prod_{i \in U} (P_n \mb{a})_i \right)^\ast \left( \prod_{i \in U} (P_n \mb{a})_i \right)} \geq 0,
\end{split}
\]
where in the next-to-last term we replaced $\alpha \circ \beta$ with $\alpha$, and $P_n$ is the symmetrization operator.
\end{proof}

\begin{Remark}
\label{Remark:Faithful}
For $\mc{D} = \mf{C}$, $\ip{\cdot}{\cdot}_n$ is essentially the induced inner product on a tensor product of Hilbert spaces, and so is non-degenerate if $\Exp{\cdot}$ is faithful. In general, $\ip{\cdot}{\cdot}_n$, and so $\phi_{\mc{INC}}$, is rarely faithful. For example, let $\mc{D} = \mc{M}$, $\Exp{a} = a$, and $p \in \mc{M}$ be a idempotent. Then $\ip{(1-p) \otimes p}{(1-p) \otimes p}_2 = 0$.
\end{Remark}

\begin{Notation}
For $\mb{a} \in \mc{M}^{\otimes n}$ and $(\pi, S) \in \mc{INC}(n)$, define the contraction $\mc{C}^{(\pi, S)}(\mb{a})$ by a linear extension of
\[
\mc{C}^{(\pi, S)}(a_1 \otimes \ldots \otimes a_n) = \prod_{U \in \pi \setminus S} \Exp{a_U} \bigotimes_{V \in S} a_V.
\]
Note that in all our examples with $\mc{D} = \mf{C}$,
\[
\KS{\mb{a}}^{(\pi, S)} = \KS{\mc{C}^{(\pi, S)}(\mb{a})}.
\]
\end{Notation}

\begin{Prop}
Assume $\Exp{\cdot}$ is a faithful state such that in its representation on $L^2(\mc{M}, \Exp{\cdot})$, $\mc{M}$ is represented by bounded operators. Let $\mb{a} \in \mc{M}^{\otimes n}$ and $\mb{b} \in \mc{M}^{\otimes k}$. Denote $\norm{\mb{a}}_2 = \sqrt{\ip{\mb{a}}{\mb{a}}_n}$ and $\norm{\KS{\mb{a}}}_\phi = \sqrt{\state{\KS{\mb{a}}^\ast \KS{\mb{a}}}}$. Denote
\[
Z_{n, k} = \set{(\pi, S) \in \mc{INC}(n+k): \pi \wedge (\hat{1}_n, \hat{1}_k) = \hat{0}_{n+k}, \Sing(\pi) \subset S}.
\]
For $(\pi, S) \in Z_{n, k}$, the map
\[
\mb{b} \mapsto \KSP{\mc{INC}}{\mb{a} \otimes \mb{b}}^{(\pi, S)}
\]
is bounded as a map from $L^2(\mc{M}, \Exp{\cdot})^{\otimes k}$ to $L^2(\mc{M}, \Exp{\cdot})^{\otimes \abs{S}}$, as is the map
\[
\mb{b} \mapsto \KSP{\mc{INC}}{\mb{a}} \KSP{\mc{INC}}{\mb{b}}.
\]
Therefore the definitions of $\mc{C}^{(\pi, S)}(\mb{a} \otimes \mb{b})$ and $\KSP{\mc{INC}}{\mb{a} \otimes \mb{b}}^{(\pi, S)}$ and the identity
\[
\KSP{\mc{INC}}{\mb{a}} \KSP{\mc{INC}}{\mb{b}}
= \sum_{(\pi, S) \in Z_{n,k}} \KSP{\mc{INC}}{\mc{C}^{(\pi, S)}(\mb{a} \otimes \mb{b})}
= \sum_{(\pi, S) \in Z_{n,k}} \KSP{\mc{INC}}{\mb{a} \otimes \mb{b}}^{(\pi, S)}
\]
extend to $\mb{a} \in \mc{M}^{\otimes n}$ (algebraic tensor product) and $\mb{b} \in L^2(\mc{M}, \Exp{\cdot})^{\otimes k}$ (Hilbert space tensor product). If $\Exp{\cdot}$ is tracial, we may switch $\mb{a}$ and $\mb{b}$.
\end{Prop}

\begin{proof}
Since
\[
\norm{\KSP{\mc{INC}}{\mb{a} \otimes \mb{b}}^{(\pi, S)}}_{\phi_{\mc{INC}}}
= \norm{\mc{C}^{(\pi, S)}(\mb{a} \otimes \mb{b})}_2,
\]
it suffices to consider the map $\mb{b} \mapsto \mc{C}^{(\pi, S)}(\mb{a} \otimes \mb{b})$. Denote
\[
\pi_\ell = \set{(i): 1 \leq i \leq n - \ell, n + \ell + 1 \leq i \leq n + k; (n - j + 1, n + j) : 1 \leq j \leq \ell},
\]
and
\[
S_\ell = \set{(i): 1 \leq i \leq n - \ell, n + \ell + 1 \leq i \leq n + k}, \quad S_\ell' = S_\ell \cup \set{(n - \ell + 1, n + \ell)}.
\]
Note $\abs{S_\ell} = n + k - 2 \ell$, $\abs{S_\ell'} = n + k - 2 \ell + 1$, and $\abs{Z_{n,k}} = 2 \min(n,k)$. Then
\[
Z_{n, k} = \set{(\pi_\ell, S_\ell), (\pi_\ell, S_\ell') : 0 \leq \ell \leq \min(n-k)}.
\]
For $\mb{a} = a_1 \otimes \ldots \otimes a_n$ and $\mb{b} = \sum_i b_{i1} \otimes \ldots \otimes b_{ik}$,
\[
\begin{split}
& \norm{\mc{C}^{(\pi_\ell, S_\ell')}(\mb{a} \otimes \mb{b})}_2^2
= \norm{\sum_i \prod_{r = 1}^{\ell - 1} \Exp{a_{n-r+1} b_{i, r}} a_1 \otimes \ldots \otimes a_{n - \ell} \otimes a_{n - \ell + 1} b_{i \ell} \otimes \ldots \otimes b_{ik}}_2^2 \\
& = \sum_{i,j} \prod_{r = 1}^{\ell - 1} \overline{\Exp{a_{n-r+1} b_{i, r}}} \prod_{s = 1}^{\ell - 1} \Exp{a_{n-s+1} b_{j, s}} \Exp{a_1^\ast a_1} \ldots \Exp{a_{n-\ell}^\ast a_{n - \ell}} \Exp{b_{i \ell}^\ast a_{n - \ell + 1}^\ast a_{n - \ell + 1} b_{j \ell}} \ldots \Exp{b_{ik}^\ast b_{jk}} \\
&\leq \norm{a_{n - \ell + 1}}^2 \norm{\mc{C}^{(\pi_{\ell - 1}, S_{\ell - 1})}(\mb{a} \otimes \mb{b})}_2^2
\end{split}
\]
Also, the statement
\[
\norm{\mc{C}^{(\pi_\ell, S_\ell)}(\mb{a} \otimes \mb{b})}_2
\leq \norm{\mb{a}}_2 \norm{\mb{b}}_2
\]
is about Hilbert spaces and not algebras, and as such is well known, see for example Proposition~5.3.3 in \cite{BiaSpeBrownian} (one may identify the Hilbert space with the space of square-integrable functions on a measure space, and apply coordinate-wise Cauchy-Schwarz inequality). The boundedness of the first map follows.

Next, note that $\KSP{\mc{INC}}{\mc{C}^{(\pi, S)}(\mb{a} \otimes \mb{b})}$ are orthogonal for different $\abs{S}$. Thus
\[
\begin{split}
& \norm{\KSP{\mc{INC}}{\mb{a}} \KSP{\mc{INC}}{\mb{b}}}_{\phi_{\mc{INC}}}^2
= \sum_{(\pi, S) \in Z_{n,k} } \norm{\KSP{\mc{INC}}{\mc{C}^{(\pi, S)}(\mb{a} \otimes \mb{b})}}_{\phi_{\mc{INC}}}^2.
\end{split}
\]
The results follow.
\end{proof}

\begin{Example}
Let $f(x,y) = \chf{[0,1]}(y) \chf{[0, y^{-1/4}]}(x)$ and $g(y) = y^{-1/4}$. Then $f \in L^1 \cap L^\infty(\mf{R}^2)$ (and so in $L^2(\mf{R}^2)$) and $g \in L^2(\mf{R})$, but
\[
\mc{C}^{(\set{(1), (2,3)}, \set{(1), (2,3)})} (f \otimes g)(x,y) = f(x,y) g(y)
\]
is not in $L^2(\mf{R}^2)$. Cf. Remark~3.3 in \cite{Bourguin-Peccati-Semicircular-free-Poisson}.
\end{Example}

\begin{Remark}
\label{Remark:Ito-isometry}
In stochastic analysis, see for example \cite{Peccati-Taqqu-book} or \cite{BiaSpeBrownian}, it is usual to prove product formulas
\begin{equation}
\label{Eq:Product}
\KS{a_1 \otimes \ldots \otimes a_n} \KS{b_1 \otimes \ldots \otimes b_k} = \sum W
\end{equation}
for all $a_i$'s, and separately all $b_j$'s, orthogonal to each other. One can then conclude using the It\^{o} isometry that the same formula holds for general $a_i, b_j$. Some, but not all, of the ingredients of this approach generalize to the Wick product setting.
\begin{itemize}
\item
In the case of $W_{\mc{INC}}$, $W$, and $W_q$, we have isometries between $\Gamma(\mc{M})$ and $\bigoplus_{n=0}^\infty \mc{M}^{\otimes n}$ with, respectively, the usual inner product induced by $\Exp{\cdot}$, the inner product \eqref{Eq:Inner-product-Meixner}, and the appropriate $q$-inner product (equation~4.73 in \cite{AnsAppell}). So in all these cases, one may extend the definition of $W$ to the appropriate closure, which however are different in all three cases.
\item
Instead of starting with general simple tensors, we could have started with the analog of functions supported away from diagonals. As noted in Lemma~\ref{Lemma:Density} below, in the infinite-dimensional setting such elements are still dense with respect to the usual inner product. However they are clearly not dense for the inner product \eqref{Eq:Inner-product-Meixner}.  For example, in the natural commutative setting of $(\mc{M}, \Exp{\cdot}) = ((L^1 \cap L^\infty(\mf{R}), \,dx)$, the inner product on functions of $n$ arguments is
\[
\sum_{\pi \in \Int(n)} \overline{f(\mb{x})} g(\mb{x}) \,d\mu_\pi(\mb{x}),
\]
where $\mu_\pi$ is a multiple of the $\abs{\pi}$-dimensional Lebesgue measure on the diagonal set
\[
\set{\mb{x} \in \mf{R}^n: x_i = x_j \Leftrightarrow i \stackrel{\pi}{\sim} j}.
\]
This is the reason why the formulas in Theorem~\ref{Thm:Free-Meixner-Wick} take a considerably simpler form if the arguments have orthogonal components.
\item
Finally, to extend the product relation \eqref{Eq:Product}, we need the product map to be continuous, at least when one of the arguments is in the algebraic tensor product and the other is bounded in two-norm. If the state $\phi$ is not tracial, this need not be the case. Since $\phi_q$ is not tracial, it is natural to expect a counterexample in Section~\ref{Sec:Counterexample}.
\end{itemize}
\end{Remark}

It is well-known that for non-atomic measures, functions supported away from diagonals are dense in the product space of all square integrable functions. The next lemma (applied to $L^2(\mc{M}, \Exp{\cdot})$. shows that this results remains true for non-commutative algebras, in fact with no assumptions on the state other than faithfulness. The result is surely known, but we could not find it in the literature.

\begin{Lemma}
\label{Lemma:Density}
Let $H$ be a Hilbert space. In the Hilbert space tensor product $H \otimes H$, consider the span $S$ of tensors of the form $f \otimes g$ with $\ip{f}{g} = 0$. This span is dense if and only if $H$ is infinite dimensional.
\end{Lemma}

\begin{proof}
Choose an orthonormal basis $\set{e_i}$ for $H$. Identify $H \otimes H$ with the space of Hilbert-Schmidt operators $\mathrm{HS}(H)$,
\[
f \otimes g \mapsto f \ip{\cdot}{g}.
\]
Then for any $f, g$, we have
\[
\tr(f \otimes g) = \sum_i \ip{f}{e_i} \ip{e_i}{g} = \ip{f}{g}.
\]
In particular, if $\ip{f}{g} = 0$, $\tr(f \otimes g) = 0$. So if $\dim H < \infty$, all operators in $S$ have trace zero, and so $S$ is not dense.

Now suppose that $\dim H = \infty$. Then $H$ is isomorphic to $L^2([0,1], dx)$, in which case the result is well-known (it is also not hard to give a direct argument in terms of Hilbert-Schmidt operators; it is left to the interested reader).
\end{proof}

\appendix

\section{Combinatorial corollaries}
\label{Sec:Combinatorial}

\begin{Prop}
All five examples of incomplete posets in Section~\ref{Sec:Posets} are graded by the number of open blocks. In addition:
\begin{enumerate}
\item
The number of incomplete partitions (analog of Bell numbers) is
\[
\abs{\mc{IP}(n)} = \sum_{i=0}^n \binom{n}{i} B_i B_{n-i},
\]
sequence A001861 in \cite{OEIS}. The incomplete Stirling numbers of the second kind are
\[
S_{n,k,\ell} = \abs{\set{(\pi, S) \in \mc{IP}(n): \abs{\pi \setminus S} = k, \abs{S} = \ell}} = \binom{k + \ell}{\ell} S_{n, k + \ell},
\]
and the number of elements of rank $\ell$ is $\sum_k \binom{k + \ell}{\ell} S_{n, k + \ell}$, sequence A049020.
\item
The number of incomplete non-crossing partitions (analog of Catalan numbers) is
\[
\abs{\mc{INC}(n)} = \binom{2n}{n},
\]
the central binomial coefficients, sequence A000984. Define the incomplete Narayana numbers
\[
N_{n,k,\ell} = \abs{\set{(\pi, S) \in \mc{INC}(n): \abs{\pi \setminus S} = k, \abs{S} = \ell}}.
\]
Then denoting
\[
F(t,x,z) = \sum_{n=0}^\infty \sum_{k=0}^\infty \sum_{\ell = 0}^\infty N_{n,k,\ell} t^k x^\ell z^n
\]
their generating function and
\[
\tilde{F}(t,z) = F(t, 0, z) = \frac{1 + z (t-1) - \sqrt{1 - 2 z (t + 1) + z^2 (t - 1)^1}}{2 t z}
\]
the generating function of the regular Narayana numbers,
\[
F(t,x,z) = \frac{1 - z \tilde{F}(t,z)}{1 - z (t + x + \tilde{F}(t,z))}.
\]
The rank generating function is
\[
F(1, x, z) = \frac{1 + \sqrt{1 - 4 z}}{1 - 2 z (x + 1) + \sqrt{1 - 4 z}},
\]
and the number of elements of rank $\ell$ is $\frac{2 \ell + 1}{n +  \ell + 1} \binom{2n}{n - \ell}$, sequence A039599.
\item
The number of partial permutations is $\abs{\mc{IPRM}(n)} = \sum_{\ell = 0}^n \binom{n}{\ell}^2 \ell!$, sequence A002720. The incomplete Stirling numbers of the first kind
\[
s_{n,k,\ell} = \abs{\set{(\pi, S) \in \mc{IPRM}(n): \abs{\pi \setminus S} = k, \abs{S} = \ell}}
\]
have the generating function
\[
\sum_{k=0}^n s_{n,k,\ell} t^k = \binom{n}{\ell} (t+\ell) \ldots (t+n-1),
\]
and the number of elements of rank $\ell$ is $\binom{n}{\ell}^2 \ell!$.
\end{enumerate}
\end{Prop}

\begin{proof}
The formula for the incomplete Stirling numbers of the second kind is obvious. Then using for example the solved Exercise~1.32 in \cite{Aigner-Course-Enumeration},
\[
\abs{\mc{IP}(n)} = \sum_{k, \ell = 0}^n \binom{k + \ell}{\ell} S_{n, k + \ell}
= \sum_{k, \ell = 0}^n \sum_{i = \ell}^{n-k} \binom{n}{i} S_{i, \ell} S_{n-i, k}
=\sum_{i=0}^n \binom{n}{i} B_i B_{n-i}.
\]
The incomplete Narayana numbers satisfy the recursion relation
\[
N_{n+1,k,\ell} = N_{n,k-1,\ell} + N_{n,k,\ell-1} + \sum_{i=1}^n \sum_{j=0}^k N_{i, j, \ell} N_{n-i,k-j,0}.
\]
It follows that
\[
F(t,x,z) = 1 + z t F(t,x,z) + z x F(t,x,z) + z (F(t,x,z) - 1) \tilde{F}(t,z).
\]
For $t=1$ this relation is easily solved to give the rank generating function, while setting additionally $x=1$, we see that the generating function for $\abs{\mc{INC}(n)}$ is
\[
F(1,1,z) = \frac{1}{\sqrt{1 - 4 z}}.
\]
Using the bijection from Proposition~\ref{Prop:INC},
\[
\abs{\set{(\pi, S) \in \mc{INC}(n): \abs{S} = \ell}} = \abs{\set{\pi \in \mc{INC}_{1,2}(2n): \abs{\Sing(\pi)} = 2 \ell}}.
\]
The latter number is clearly the same as the number of lattice paths with W and N steps which go from $(0,0)$ to $(n + \ell, n - \ell)$ and do not cross the main diagonal. Using the reflection principle, this number is $\binom{2n}{n + \ell} - \binom{2n}{n + \ell + 1}$.

The formula for $\abs{\mc{IPRM}(n)}$ is obvious. The formula for the incomplete Stirling numbers of the first kind follows from the recursion relation
\[
s_{n+1,k,\ell} = s_{n,k-1,\ell} + s_{n,k,\ell-1} + (n + \ell) s_{n,k,\ell},
\]
obtained in the usual way by adjoining $n+1$ to an incomplete permutation of $n$; note that in a closed work of length $u$, $n+1$ can be inserted in $u$ places, while in an open word it can be inserted in $u+1$ spaces.
\end{proof}

\begin{Remark}
For completeness, we include combinatorial corollaries of Proposition~\ref{Prop:classical-expansions-2} and Theorem~\ref{Thm:5th-expansion}. Take $a$ to a projection, so that $a^2 = a$ and $\Exp{a} = t$. Denote $X(a) = x$. Then
\[
x \KSP{\mc{IP}_{1,2}}{a^{\otimes n}} = \KSP{\mc{IP}_{1,2}}{a^{\otimes {n+1}}} + t n \KSP{\mc{IP}_{1,2}}{a^{\otimes {n-1}}},
\]
\[
x \KSP{\mc{IP}}{a^{\otimes n}} = \KSP{\mc{IP}}{a^{\otimes {n+1}}} + (t + n) \KSP{\mc{IP}}{a^{\otimes {n}}} + t n \KSP{\mc{IP}}{a^{\otimes {n-1}}},
\]
and
\[
\begin{split}
x \KSP{\mc{IPRM}}{a^{\otimes n}}
& = \KSP{\mc{IPRM}}{a^{\otimes {n+1}}} + (t + 2 n) \KSP{\mc{IPRM}}{a^{\otimes {n}}} \\
& \quad + (t n + n (n-1)) \KSP{\mc{IPRM}}{a^{\otimes {n-1}}}.
\end{split}
\]
Thus $\KSP{\mc{IP}_{1,2}}{a^{\otimes n}} = H_n(x,t)$, the Hermite polynomial; $\KSP{\mc{IP}}{a^{\otimes n}} = C_n(x,t)$, the Charlier polynomial; and $\KSP{\mc{IPRM}}{a^{\otimes n}} = L_n^{(t-1)}(x)$, the Laguerre polynomial. We thus get (mostly known) inversion, moment, product, and linearization formulas for these polynomials. For example, for the Laguerre case
\begin{equation*}
x^n = \sum_{(\pi, S) \in \mc{IPRM}(n)} t^{\abs{\pi \setminus S}} L_{\abs{S}}^{(t-1)}(x)
= \sum_{\ell=0}^n \binom{n}{\ell} (t+\ell) \ldots (t+n-1) L_{\abs{S}}^{(t-1)}(x),
\end{equation*}
\[
L_n^{(t-1)}(x) = \sum_{(\pi, S) \in \mc{IPRM}(n)} (-1)^{n - \abs{S}} t^{\abs{\pi \setminus S}} x^{\abs{S}}
= \sum_{\ell=0}^n (-1)^{n-\ell} \binom{n}{\ell} (t+\ell) \ldots (t+n-1) x^\ell,
\]
and
\[
\prod_{i=1}^k L_{s(i)}^{(t-1)}(x)
= \sum_{(\pi, S) \in \mc{ID}(s(1), \ldots, s(k))} t^{\abs{\pi \setminus S}} L_{\abs{S}}^{(t-1)}(x).
\]
extending the moment and linearization formulas \cite{Foata-Zeilberger-Laguerre}
\[
\int x^n \,d\mu(x) = \sum_{\pi \in \Sym(n)} t^{\cyc{\pi}} = t (1+1) \ldots (t+n-1),
\]
\[
\int \prod_{i=1}^k L_{s(i)}^{(t-1)}(x) \,d\mu(x) = \sum_{\pi \in \mc{D}(s(1), \ldots, s(k))} t^{\cyc{\pi}}.
\]
Similarly, since
\[
\sum_{\substack{(\pi, S) \in \mc{IP}(n), U \in \pi \setminus S \Rightarrow \abs{U} = 1, \\ \abs{\pi \setminus S} = k, \abs{S} = \ell}} \prod_{V \in S} (\abs{V} - 1)!
= \binom{n}{k} \sum_{\pi \in \mc{P}(n-k), \abs{\pi} = \ell} \prod_{V \in \pi} (\abs{V} - 1)!
= \binom{n}{k} s_{n-k,\ell},
\]
we obtain the familiar result that the Charlier polynomials are
\[
C_n(x,t) = \sum_{k, \ell = 0}^n (-1)^{n - \ell} \binom{n}{k} s_{n-k,\ell} t^k x^\ell
= \sum_{k=0}^n (-1)^{k} \frac{n!}{k!} t^k \binom{x}{n-k}.
\]
Finally, since
\[
\abs{\set{(\pi, S) \in \mc{INC}(n), \pi \in \Int(n), V \in \pi \setminus S \Rightarrow \abs{V} = 1, \abs{\pi \setminus S} = k, \abs{S} = \ell}}
= \binom{n-k}{\ell - 1} \binom{k + \ell}{\ell},
\]
the free Charlier polynomials are
\[
P_n(x,t) = \sum_{k, \ell} (-1)^{n - \ell} \binom{n-k}{\ell - 1} \binom{k + \ell}{\ell} t^k x^\ell.
\]
See, for example, Chapter 7 in \cite{Aigner-Course-Enumeration} for many related combinatorial results.
\end{Remark}

\begin{Remark}
\label{Remark:Meixner-moments}
Let $\mu_{\alpha, \beta, t, \gamma}$ be the measure of orthogonality of the free Meixner polynomials, with the Jacobi-Szeg\H{o} parameters
\[
\begin{pmatrix}
\alpha, & \alpha + \beta, & \alpha + \beta, & \ldots \\
t, & t + \gamma, & t + \gamma, & \ldots
\end{pmatrix}.
\]
Then from the Viennot-Flajolet theorem, the $n$'th moment of this measure is
\begin{multline}
\label{Eq:Viennot}
\sum_{\tau \in \mc{NC}_{1,2}(n)} \prod_{\substack{V \in \Outer(\tau), \\ \abs{V} = 1}} \alpha \prod_{\substack{V \in \tau \setminus \Outer(\tau), \\ \abs{V} = 1}} (\alpha + \beta) \prod_{\substack{U \in \Outer(\tau), \\ \abs{U} = 2}} t \prod_{\substack{U \in \tau \setminus \Outer(\tau), \\ \abs{U} = 2}} (t + \gamma) \\
= \sum_{\sigma \in \mc{NC}(n)}  \alpha^{\abs{\Sing(\sigma)}} \beta^{n - 2 \abs{\sigma} + \abs{\Sing(\sigma)}} t^{\abs{\Outer(\sigma) \setminus \Sing(\sigma)}} (t + \gamma)^{\abs{\sigma} - \abs{\Sing(\sigma)} - \abs{\Outer(\sigma) \setminus \Sing(\sigma)}}.
\end{multline}
We may interpret this as saying that the two-state free cumulants of the pair of free Meixner and free Poisson measures $(\mu_{\alpha, \beta, t, \gamma}, \mu_{\alpha, \beta, 0, t})$ are
\[
R_1 = r_1 = \alpha, \quad R_j = (t + \gamma) \beta^{j-2}, \quad r_j = t \beta^{j-2}.
\]
Cf. Proposition~10 in \cite{AnsAppell3}.

Various classical combinatorial sequences appearing as moments of these measures are listed in Section~7.4 of \cite{Aigner-Course-Enumeration}. These include Catalan, Motzkin, and Schr{\"o}der numbers. Expansions \eqref{Eq:Moment} and \eqref{Eq:Viennot} then give us various combinatorial identities. For example, for $\alpha =  t = \gamma = 1$ and $\beta = 2$, the free cumulants are Catalan numbers while the moments are the (shifted) large Schr{\"o}der numbers, and we obtain the relations
\[
\mathrm{Sch}_{n-1} = \sum_{\pi \in \mc{NC}(n)} \prod_{U \in \pi} c_{\abs{U} - 1}
= \sum_{\sigma \in \mc{NC}(n)} 2^{n - \abs{\sigma} - \abs{\Outer(\sigma) \setminus \Sing(\sigma)}}.
\]
For the first relation, cf. Corollary~8.4 in \cite{Dykema-Multilinear-function-series}. If $\beta = t = \gamma = 1$, and $\alpha = 0$ the free cumulants are Motzkin numbers, and the moments are
\[
\sum_{\pi \in \mc{NC}_{\geq 2}(n)} \prod_{U \in \pi} M_{\abs{U} - 2}
= \sum_{\sigma \in \mc{NC}_{\geq 2}(n)} 2^{\abs{\sigma} - \abs{\Outer(\sigma)}}.
\]
Either for $\alpha = 1$ or $\alpha = 0$ this moment sequence does not appear in \cite{OEIS}.
\end{Remark}

\begin{Remark}
In this remark we compute the sum of the coefficients in the expansion \eqref{Eq:Monomial}. According to Lemma~\ref{Lemma:Kappa}, this sum is
\[
T_n = \sum_{(\pi, S) \in \mc{INC}(n)} \prod_{U \in \pi \setminus S} \kappa_{\alpha, \beta, t, \gamma}^{\abs{U}} \prod_{V \in S} \omega_{\alpha, \beta, t, \gamma}^{\abs{V}}
= \sum_{\pi \in \mc{NC}(n)} \prod_{U \in \pi \setminus \Outer(\pi)} \kappa_{\alpha, \beta, t, \gamma}^{\abs{U}} \prod_{V \in \Outer(\pi)} (\omega_{\alpha, \beta, t, \gamma}^{\abs{V}} + \kappa_{\alpha, \beta, t, \gamma}^{\abs{V}}).
\]
Using the same lemma,
\[
r(z) = \sum_{n=1}^\infty \kappa_{\alpha, \beta, t, \gamma}^n z^{n-1}
= \alpha + t \sum_{n=2}^\infty M_{n-2}(\beta, \gamma) z^{n-1}
= \alpha + t z F_{\beta, \gamma}(z)
\]
and
\[
R(z) = \sum_{n=1}^\infty (\omega_{\alpha, \beta, t, \gamma}^n + \kappa_{\alpha, \beta, t, \gamma}^n) z^{n-1}
= \alpha + t z F_{\beta, \gamma}(z) + \sum_{n=1}^\infty M_{n-1}(\beta, \gamma) z^{n-1}
= \alpha + (t z + 1) F_{\beta, \gamma}(z)
\]
where
\[
F_{\beta, \gamma}(z) = \sum_{n=0}^\infty M_n(\beta, \gamma) z^n
= \frac{1 - \beta z - \sqrt{(\beta z - 1)^2 - 4 \gamma z^2}}{2 \gamma z^2}
\]
is the generating function of Motzkin polynomials. According to the two-state free probability theory, the moment generating function of $\mu_{\alpha, \beta, t, \gamma}$ from Remark~\ref{Remark:Meixner-moments} is the solution of
\[
\frac{1}{z m_{\alpha, \beta, t, \gamma}(z)} + r(z m_{\alpha, \beta, t, \gamma}(z)) = \frac{1}{z},
\]
and the generating function of the desired sequence is
\[
M(z) = \sum_{n=0}^\infty T_n z^n = \frac{1}{1 - z R(z m_{\alpha, \beta, t, \gamma}(z))}.
\]
According to Remark \ref{Remark:Other-bijections} below, the most natural choice of the parameters appears to be $\alpha = t = \gamma = 1$, $\beta = 2$. In this case
\[
F(z) = \frac{1 - 2 z - \sqrt{1 - 4 z}}{2 z^2}, \qquad R(z) = \frac{1 - z - (z + 1) \sqrt{1 - 4 z}}{2 z^2},
\]
and $m(z)$ is the generating function of shifted large Schr{\"o}der numbers
\[
m(z) = \frac{3 - z - \sqrt{1 - 6 z + z^2}}{2}.
\]
Using Maple, we compute the first few terms in the sequence $T_n$ to be $1, 2, 7, 30, 140, 684$. This sequence now appears in \cite{OEIS} as A299296, but has not arisen in other contexts.
\end{Remark}

\begin{Example}
From Theorem~\ref{Thm:Inversion}, we can get a variety of different-looking combinatorial expansions.

For $\alpha = \beta = \gamma = t = 1$, Case II. $u, v = e^{\pm (\pi/3) i}$.
\[
o_k = 1, 1, 0, -1, -1, 0, \ldots, \qquad c_k = 1, 0, -1, -1, 0, 1, \ldots.
\]
\begin{multline*}
\KS{a_1 \otimes \ldots \otimes a_n} \\
= \sum_{\substack{(\pi, S) \\ \pi \in \Int(n) \\ S \subset \pi \\ U \in \pi \setminus S \Rightarrow \abs{U} \neq 2 \mod 3 \\ V \in S \Rightarrow \abs{V} \neq 0 \mod 3}} (-1)^{n - \abs{U \in \pi \setminus S: \abs{U} = 3 \text{ or } 4 \text{ mod } 6} - \abs{V \in S: \abs{V} = 1 \text{ or } 2 \text{ mod } 6}} \prod_{U \in \pi \setminus S} \Exp{a_U} \prod_{V \in S} X(a_V).
\end{multline*}
For $\alpha = 1$, $\gamma = t$, $\beta = t + 1$, Case II$'$. $u=1$, $v=t$.
\[
o_k = \frac{1}{1 - t} (1 - t^k), \qquad c_k = 1.
\]
\[
\KS{a_1 \otimes \ldots \otimes a_n} = \sum_{\substack{(\pi, S) \\ \pi \in \Int(n) \\ S \subset \pi}} (-1)^{n - \abs{S}} \prod_{U \in \pi \setminus S} \Exp{a_U} \prod_{V \in S} \left( \frac{1 - t^{\abs{V}}}{1 - t} \right) X(a_V).
\]
For $\alpha = 0$, $\gamma = t = 1$, $\beta = 2$, Case III.
\[
o_k = k, \qquad c_k = -(k-1).
\]
\[
\KS{a_1 \otimes \ldots \otimes a_n} = \sum_{\substack{(\pi, S) \\ \pi \in \Int(n) \\ \Sing(\pi) \subset S \subset \pi}} (-1)^{n - \abs{\pi}} \prod_{U \in \pi \setminus S} (\abs{U} - 1) \Exp{a_U} \prod_{V \in S} \abs{V} X(a_V).
\]
For $\alpha = \gamma = t = 1$, $\beta = 2$, Case III$'$.
\[
o_k = k, \qquad c_k = 1.
\]
\[
\KS{a_1 \otimes \ldots \otimes a_n} = \sum_{\substack{(\pi, S) \\ \pi \in \Int(n) \\ S \subset \pi}} (-1)^{n - \abs{S}} \prod_{U \in \pi \setminus S} \Exp{a_U} \prod_{V \in S} \abs{V} X(a_V).
\]
These in turn give expansions for free Meixner polynomials and may serve as a source of combinatorial identities.
\end{Example}

\section{Alternative approaches to Theorem~\ref{Thm:Free-Meixner-Wick}}
\label{Sec:Alternative}

\begin{Remark}
An alternative combinatorial structure we could have used in Section~\ref{Sec:Free-Meixner} are linked partitions. According to \cite{Dykema-Multilinear-function-series,Nica-NC-Linked}, the pairs $\set{\sigma \ll \pi: \sigma, \pi \in \mc{NC}(n)}$ are in a natural bijection with the set of non-crossing linked partitions $\mc{NCL}(n)$, and doubling the value of $\beta$ gives a bijection between such pairs with $\Sing(\sigma) = \Sing(\pi)$ and all such pairs. Moveover, according to \cite{Chen-Wu-Yan-Linked-partitions,Chen-Liu-Wang-Linked-partitions-tableaux}, permutations are in a natural bijection with the set of all linked partitions $\mc{LP}(n)$. The results of Theorems~\ref{Thm:5th-expansion} and \ref{Thm:Free-Meixner-Wick} can be phrased in terms of these objects, see \cite{Yano-Yoshida-Linked-partitions} for related moment computations. This approach has not led us to any clarification in the inversion or product formulas.

In place of partitions, we could also (of course) have used colored Motzkin paths. From the point of view of Definition~\ref{Defn:Free-Meixner-Wick}, the most natural family are those with a single color for rising steps and flat and falling steps at height zero, two colors for the other falling steps, and three colors for the rest of flat steps. It is not hard to see using the continued fraction form of the generating functions that the number of such paths of length $n+1$ is equal to the number of large $(3,2)$-Motzkin paths of length $n$ in the sense of \cite{Chen-Wang-NC-linked-Motzkin} (similar to the above, except their falling and flat steps at height zero are allowed two colors). This number in turn is known to be the (large) Schr{\"o}der number, see Remark~\ref{Remark:Meixner-moments}.
\end{Remark}

\begin{Remark}
\label{Remark:Other-bijections}
Unlike in the expansions in the five examples in Section~\ref{Sec:Wick}, the terms on the right hand side of \eqref{Eq:Monomial} have multiplicities. One can modify Definition~\ref{Defn:Free-Meixner-Wick} to obtain bijective representations. For example, we may define instead
\begin{equation*}
\begin{split}
\KS{a_1 \otimes \ldots \otimes a_n \otimes a_{n+1}}
& = \KS{a_1 \otimes \ldots \otimes a_n} X(a_{n+1}) - \alpha \KS{a_1 \otimes \ldots \otimes a_n} \Exp{a_{n+1}} \\
&- \beta \KS{a_1 \otimes \ldots \otimes a_{n-1} \otimes a_n a_{n+1}} - t \KS{a_1 \otimes \ldots \otimes a_{n-1}} \Exp{a_n a_{n+1}} \\
&- \gamma \KS{a_1 \otimes \ldots \otimes a_{n-2} \otimes a_n a_{n+1} a_{n-1}}.
\end{split}
\end{equation*}
Note that this definition works only in the scalar-valued and not in the operator-valued case. The corresponding terms are in a bijection with the following collection of incomplete permutations. First, they have no double descents. Second, arrange each closed block so that it ends in its largest element. Then the descent-ascents in each block appear in decreasing order. Finally, split each block into sub-words, ending with the final letter or a descent-ascent, and beginning with the initial letter or right after the preceding descent-ascent. Then the partition into these sub-words is non-crossing.

We may also define
\begin{equation*}
\begin{split}
\KS{a_1 \otimes \ldots \otimes a_n \otimes a_{n+1}}
& = \KS{a_1 \otimes \ldots \otimes a_n} X(a_{n+1}) - \alpha \KS{a_1 \otimes \ldots \otimes a_n} \Exp{a_{n+1}} \\
&\quad - \beta_1 \KS{a_1 \otimes \ldots \otimes a_{n-1} \otimes a_n a_{n+1}} \\
&\quad - \beta_2 \KS{a_1 \otimes \ldots \otimes a_{n-1} \otimes a_{n+1} a_{n}} \\
&\quad - t \KS{a_1 \otimes \ldots \otimes a_{n-1}} \Exp{a_n a_{n+1}} \\
&\quad - \gamma \KS{a_1 \otimes \ldots \otimes a_{n-2} \otimes a_n a_{n+1} a_{n-1}},
\end{split}
\end{equation*}
The corresponding terms are in a bijection with the following collection of incomplete permutations. Arrange each closed block so that it ends in its largest element. Then the descent-ascents in each block appear in decreasing order. Split each block as above. Then the partition into these sub-words is non-crossing, and on each sub-block, the letters are decreasing and then increasing, with the sub-block maximum at the end.

This description appears related to the work of West \cite{West-Generating-trees}, who studied permutations avoiding the patters $(3142, 2413)$ (sometimes called separable permutations). He proved that the cardinality of this set is the Schr{\"o}der number (see Remark~\ref{Remark:Meixner-moments}), and the argument uses trees reminiscent of the construction above.
\end{Remark}


\def\cprime{$'$} \def\cprime{$'$}
\providecommand{\bysame}{\leavevmode\hbox to3em{\hrulefill}\thinspace}
\providecommand{\MR}{\relax\ifhmode\unskip\space\fi MR }
\providecommand{\MRhref}[2]{%
  \href{http://www.ams.org/mathscinet-getitem?mr=#1}{#2}
}
\providecommand{\href}[2]{#2}

\end{document}